\documentclass{amsart}

\usepackage{hyperref}
\usepackage{amsfonts,amsmath,amsthm,amssymb}
\usepackage{graphicx,tikz,float}
\usepackage{enumerate}
\usepackage{todonotes}
\usepackage{booktabs,array}
\usepackage{microtype,fullpage}
\usepackage[utf8]{inputenc}
\usepackage{pdfpages}
\usepackage{float}
\usepackage{caption}
\usepackage{subcaption}
\usepackage{amsmath,calc}
\usepackage{comment}
\usepackage[pagewise]{lineno}

\newtheoremstyle{dotless}{}{}{\itshape}{}{\bfseries}{}{ }{}
\theoremstyle{dotless}
\newtheorem{thm}{Theorem}[section]
\newtheorem{lem}{Lemma}[section]
\newtheorem{prop}{Proposition}[section]
\newtheorem{defn}{Definition}[section]

\newcommand{\Var}{\mathrm{Var}}

\theoremstyle{definition}

\theoremstyle{remark}
\newtheorem{remark}{Remark}[section]
\newtheorem*{acknowledgement}{Acknowledgement}

\newcommand{\Cauchy}{\text{Cauchy}}
\newcommand{\Bernoulli}{\text{Bernoulli}}

\begin{document}

\title{Lyapunov exponent and variance in the CLT
for products of random matrices related to random Fibonacci sequences}

	\author[Majumdar]{Rajeshwari Majumdar{$^{\dag}$}}
	\thanks{\footnotemark {$\dag$} Research was supported in part by NSF Grant DMS-1262929}
	\address{\parbox{\linewidth}{Department of Politics\\
		New York University\\
    New York, NY 10012, U.S.A}}
	\email{majumdar@nyu.edu}
	
	\author[Mariano]{Phanuel Mariano{$^{\ddag}$}}
	\thanks{\footnotemark {$\ddag$} Research was supported in part by NSF Grant DMS-1405169, 1712427.}
	\address{\parbox{\linewidth}{Department of Mathematics and Physics\\
		University of New Haven\\
		West Haven, CT 06516,  U.S.A.}}
	\email{pmariano@newhaven.edu}

\author[Panzo]{Hugo Panzo{$^{\star}$}}
\thanks{\footnotemark {$\star$} Research was supported in part by the UConn Mathematics Department and a Zuckerman fellowship.}
\address{\parbox{\linewidth}{Faculties of Electrical Engineering and Mathematics\\
Technion -- Israel Institute of Technology\\
Haifa 32000, Israel}} 
\email{panzo@campus.technion.ac.il}
	
	\author[Peng]{Lowen Peng{$^{\dag}$}}
	\address{\parbox{\linewidth}{Department of Mathematics\\
		University of Connecticut\\
		Storrs, CT 06269,  U.S.A.}}
	\email{lowen.peng@uconn.edu}

	\author[Sisti]{Anthony Sisti{$^{\dag}$}}
	\address{\parbox{\linewidth}{ Department of Biostatistics\\
		Brown University\\
		Providence, RI 02912,  U.S.A.}}
	\email{anthony\_sisti@brown.edu}


\date{\today}

\begin{abstract}

We consider three matrix models of order 2 with one random entry $\epsilon$ and the other three entries being deterministic. In the first model, we let 
$\epsilon\sim\textrm{Bernoulli}\left(\frac{1}{2}\right)$. For this model we develop a new technique to obtain
estimates for the top Lyapunov exponent in terms of a multi-level recursion involving Fibonacci-like sequences. This in turn gives a new characterization for the Lyapunov exponent in terms of these sequences. In the second model, we give similar estimates when $\epsilon\sim\textrm{Bernoulli}\left(p\right)$ and $p\in [0,1]$ is a parameter. Both of these models are related to random Fibonacci sequences. In the last model, we compute the Lyapunov exponent exactly when the random entry is replaced with $\xi\epsilon$ where $\epsilon$ is a standard Cauchy random variable and $\xi$ is a real parameter. We then use Monte Carlo simulations to approximate the variance in the CLT for both parameter models.
\end{abstract}

\maketitle

\tableofcontents

\section{Introduction}

The main purpose of our paper is to develop new methods to obtain precise estimates of Lyapunov exponents and the variance for the CLT related to the products of random matrices. Let $\{Y_i\}_{i\ge1}$ be a sequence of i.i.d. random matrices distributed according to a probability measure $\mu$. Further, let $S_n=Y_n Y_{n-1}\cdots Y_2 Y_1$. Assuming that $\mathbb E\left[\log^+\|Y_1\|\right]<\infty$, the \emph{top Lyapunov exponent} $\lambda$ associated with $\mu$ is given by
\begin{equation}\label{lyapdefn}
\lambda:=\lim_{n\to\infty}\frac1n\mathbb E\big[\log\|S_n\|\big]
\end{equation}	
	with 
	$\lambda\in\mathbb R\cup\{-\infty\}$. The top Lyapunov exponent gives the rate of exponential growth of the matrix norm of $S_n$ as $n\to\infty$. Since all finite-dimensional norms are equivalent, $\lambda$ is independent of the choice of norm $\|\cdot\|$. Although $\lambda$ depends on $\mu$, we usually omit this dependence from our notation. While one can also define a spectrum of Lyapunov exponents, in this paper we will only be concerned with the top Lyapunov exponent $\lambda$ and we refer to it as simply the Lyapunov exponent. Occasionally, when we are considering $\lambda$ over a family of distributions parametrized by some variable, we will write $\lambda$ as a function of that variable.  
	
	Furstenberg and Kesten (1960) and Le Page (1982) found analogues of
the Law of Large Numbers and Central Limit Theorem, respectively, for the norm
of these partial products. Despite these results having been established for some time, in most cases it is still impossible to compute the Lyapunov exponent explicitly from the distribution of the matrices. Moreover, computing the variance in the CLT has received scant attention in the literature. We point out that because of the difficulty in computing Lyapunov exponents,  most authors need to develop new techniques for specific matrix models rather than work in a general framework.

In this paper, we investigate the behavior of the Lyapunov exponent as the common distribution of the sequence of random matrices varies with a parameter. 
While there are works in the literature where explicit expressions have been obtained for some matrix models under certain conditions \cite{bougerol,chassaingetal,newman,limarahibe,mtwInvariant,Newman1986}, besides a few special examples, it is not possible to find a general explicit formula for the Lyapunov exponent. 
There is, however, an extensive literature on approximating the Lyapunov exponent for models where it cannot be calculated explicitly (see \cite{Protasov-Jungers-2013,pollicott}). For instance, in \cite{pollicott}, $\lambda$ is expressed in terms of associated complex functions and a more general algorithm to numerically approximate $\lambda$ is given. The method is efficient and converges very fast. The method also applies to a large class of matrix models.  There is also a significant interest in computing Lyapunov exponents in physics, with some recent work found in \cite{Akemann-Burda-Kieburg2014,Akemann-Kieburg-Wei-2013,Forrester-2013,Forrester-2015,Kargin-2014,Kieburg-Kosters-2019}. The analytic properties of the Lyapunov exponent as a function of the transition probabilities are studied in \cite{Peres-1991,Peres-1992,Ruelle-1979}. Lyapunov exponents are also useful in mathematical biology in the study of population dynamics. 

A random Fibonacci sequence $g_0,g_1,g_2\dots$ is defined by $g_0=g_1=1$ along with the recursive relation ${g_{n+1} = g_n \pm g_{n-1}}$ (linear case) or ${g_{n+1} = |g_n \pm g_{n-1}|}$ (non-linear case) for all $n\in\mathbb{N}$, where the sign $\pm$ is chosen by tossing a fair or biased coin (positive sign has probability $p$). In \cite{Viswanath2000}, Viswanath studied the exponential growth of $|g_n|$ as $n\to\infty$ in the linear case with $p=\frac{1}{2}$ by connecting it to a product of random matrices and then employing a new computational method to calculate the Lyapunov exponent to any degree of accuracy. The method involves using Stern-Brocot sequences, Furstenberg's Theorem (see Theorem \ref{furstenberg}) and the invariant measure to compute $\lambda$. We also point to the work of \cite{Janvresse-Rittaud-Rue2010,Janvresse-Rittaud-Rue2009,Janvresse-Rittaud-Rue2008} where the authors generalized the results of Viswanath by letting $0<p\leq 1$ and treating $\lambda$ as a function of $p$ which bears some similarity to the model we study in Section \ref{bernp}. They also considered the non-linear case. 

The model that is most relevant to our results is given in \cite{Goswami2004}, where the authors give an explicit formula for the cumulative distribution function of a random variable $X_p$ on $(0,\infty)$ characterized by the distributional identity
$$ X_p \sim \frac{1}{X_p}+ \epsilon_p,$$ where $\epsilon_p$ is a $\Bernoulli\left(p\right)$ random variable independent of $X_p$. Let CDF denote the cumulative distribution function for a random variable. The CDF of $X_p$ is given in terms of a continued fraction expansion. We will later see that the distribution of $X_p$ is the invariant distribution for the product of random matrices studied in Section \ref{bernp}.

We summarize the main results of the paper as follows. Consider the random matrices $$
Y_i = \left(\begin{array}{cc}
\epsilon_i & 1\\
1 & 0
\end{array}\right),
$$
where $\epsilon_i$ are i.i.d. random variables.
\begin{enumerate}

\item {\bf Lyapunov exponent when $\epsilon\sim \Bernoulli\left(p\right)$} (See Theorem \ref{UBprop}): The Lyapunov exponent $\lambda(p)$  can be estimated by 
\[
\frac{p\log 3}{4-p}\leq\lambda(p)\leq\frac{p\log 3}{2}.
\]

\item {\bf Lyapunov exponent when $\epsilon\sim \Bernoulli\left(\frac{1}{2}\right)$}(See Theorem \ref{thm:bounds}): The Lyapunov exponent $\lambda$  can be estimated by 
$$
p_n\leq\lambda\leq q_n,
$$
where 
$$
p_n =\frac{\log c_{n}}{\left(n+7\right)2^{n}} \,\,\, \mbox{and}\,\,\, q_n = \frac{\log c_{n}}{\left(n+4\right)2^{n}},
$$
and $c_n$ is given by Definitions \ref{def:recursion} and \ref{def:c_n}. Moreover, 
\[
\lim_{n\to \infty}p_{n} = \lim_{n\to\infty} q_{n} = \lambda.
\]
The method we develop differs from that of the papers listed above and requires the study of an interesting multi-level recursion satisfied by $c_n$.

\item {\bf Exact Lyapunov exponent involving Cauchy random variable} (See Proposition \ref{CauchyProp}): When

$$Y_i=
\left(\begin{array}{cc}
\xi \epsilon & -1\\
1 & 0
\end{array}\right), \ \epsilon\sim \Cauchy \left(0,1\right), \ \xi\in\mathbb{\mathbb{R}}, \ \xi\neq 0,
$$
then the Lyapunov exponent $\lambda(\xi)$ is given by
\[
\lambda(\xi)=\log\left(\frac{|\xi|+\sqrt{\xi^{2}+4}}{2}\right).
\]

\item {\bf Variance Simulation} (See Figures \ref{fig:BernPVars}, \ref{fig:Cauchy20Vars} and \ref{fig:CauchyZoomVars})
\end{enumerate}

The paper is organized as follows. In Section \ref{preliminaries} we give the preliminaries needed for the paper. In Section \ref{bernp}, we provide exact upper and lower bounds on the Lyapunov exponent associated with the product of random matrices where one entry is $\Bernoulli\left(p\right)$ with $0<p<1$. In particular, in Section \ref{bernhalf} we study the $p=\frac{1}{2}$ case and provide a sequence of progressively better bounds. We prove that these bounds converge to the Lyapunov exponent which gives a new characterization for the Lyapunov exponent. Not surprisingly, these bounds are related to Fibonacci sequences as in the work of \cite{Goswami2004,Janvresse-Rittaud-Rue2010,Janvresse-Rittaud-Rue2009,Janvresse-Rittaud-Rue2008,Viswanath2000}.

In Section \ref{xicauchy}, we give an example of a well-known model where we can calculate the Lyapunov exponent explicitly. In this model, one entry in the random matrix has the Cauchy distribution. In Section \ref{variance}, we examine the less studied variance associated with a multiplicative Central Limit Theorem for products of random matrices. The multiplicative CLT holds under some reasonable assumptions, see \cite{bougerol}. It states that for $\mathbf{x}\in\mathbb{R}^d\setminus\{\mathbf{0}\}$,
\[
\frac{1}{\sqrt{n}}\left(\log\left\Vert S_{n}\mathbf{x}\right\Vert -n\lambda\right)~\mbox{ and }~\frac{1}{\sqrt{n}}\left(\log\left\Vert S_{n}\right\Vert -n\lambda\right)
\]
converge weakly to a Gaussian random variable with mean $0$ and variance $\sigma^2>0$ as $n\to\infty$. In the special case where the distribution of $\|Y_1\mathbf{x}\|/\|\mathbf{x}\|$ doesn't depend on $\mathbf{x}\in\mathbb{R}^d\setminus\{\mathbf{0}\}$, Cohen and Newman \cite{newman} gave the explicit formulas
\begin{equation}\label{eq:newman}
\lambda=\mathbb{E}\left[\log\left(\frac{\|Y_1\mathbf{x}\|}{\|\mathbf{x}\|}\right)\right]~\text{ and }~\sigma^2=\mathbb{E}\left[\left(\log\left(\frac{\|Y_1\mathbf{x}\|}{\|\mathbf{x}\|}\right)-\lambda\right)^2\right]
\end{equation}
that hold whenever the expectations are finite. As far as the authors know, this is the only case where an explicit formula for the variance is given. Compared to the calculation of the Lyapunov exponent, there have been relatively few attempts to explicitly compute or numerically approximate the variance. We address this deficiency in the context of the parameter models that we consider by first describing an easy to implement Monte Carlo simulation scheme and then using it to approximate the variance for some of the models we considered earlier in the paper.   



\section{Preliminaries}\label{preliminaries}

In what follows, we introduce notational conventions and terminology and recall well-known results regarding the Lyapunov exponent. Let $\mathbb{P}^{1}\left(\mathbb{R}\right)$ denote the one-dimensional projective space. Recall that we can regard $\mathbb{P}^{1}\left(\mathbb{R}\right)$ as the space of all one dimensional subspaces of $\mathbb{R}^{2}$. To describe $\mathbb{P}^{1}\left(\mathbb{R}\right)$, let us first define the following equivalence relation $\sim$ on $\mathbb{R}^{2}\backslash\left\{ \mathbf{0}\right\} $.
We say that the vectors $\mathbf{x},\mathbf{x'}\in\mathbb{R}^{2}\backslash\left\{ \mathbf{0}\right\} $ are equivalent, denoted by $\mathbf{x}\sim\mathbf{x'}$, if there exists a nonzero real number $c$ such that $\mathbf{x}=c\mathbf{x'}$.
We define $\bar{\mathbf{x}}$ to be the equivalence class of a vector
$\mathbf{x}\in\mathbb{R}^{2}\backslash\left\{ \mathbf{0}\right\}$. Now we can define $\mathbb{P}^{1}\left(\mathbb{R}\right)$
as the set of all such equivalence classes $\bar{\mathbf{x}}$. We can also define a bijective map $\phi:\mathbb{P}^{1}\left(\mathbb{R}\right)\to\mathbb{R}\cup\left\{ \infty\right\} $
 by 
\[
\phi\left(\bar{\mathbf{x}}\right)=\begin{cases}
\frac{x_{1}}{x_{2}} & \mbox{if }x_{2}\neq0\\
\infty & \mbox{if }x_{2}=0
\end{cases}
\]
where $\mathbf{x}=\left(\begin{array}{c}
x_{1}\\
x_{2}
\end{array}\right)\in\mathbb{R}^{2}\backslash\left\{ \mathbf{0}\right\} $ is in the equivalence class $\bar{\mathbf{x}}$. Hence with a slight abuse of notation we can identify $\mathbb{P}^1(\mathbb{R})$ with $\mathbb{R}\cup\left\{ \infty\right\}$.

Consider the following group action of $\mathrm{GL}(2,\mathbb R)$ on $\mathbb{P}^{1}\left(\mathbb{R}\right)$. For $A=\begin{pmatrix}a&b\\c&d\end{pmatrix}\in\mathrm{GL}(2,\mathbb R)$ and $x\in\mathbb{P}^{1}\left(\mathbb{R}\right)$, we define
\[A\cdot x=\frac{ax+b}{cx+d}.\]
Let $\mu$ and $\nu$ be probability measures on $\mathrm{GL}(2,\mathbb R)$ and $\mathbb{P}^{1}\left(\mathbb{R}\right)$, respectively. We say  that $\nu$ is \emph{$\mu$-invariant} if it satisfies
\begin{equation}\label{eq:invariance}
\int_{\mathbb{P}^{1}\left(\mathbb{R}\right)}f(x)\,\mathrm d\nu(x)=\int_{\mathbb{P}^{1}\left(\mathbb{R}\right)}\int_{\mathrm{GL}(2,\mathbb R)}f(A\cdot x)\,\mathrm d\mu(A)\,\mathrm d\nu(x)
\end{equation}
for all bounded measurable functions $f:\mathbb{P}^{1}\left(\mathbb{R}\right)\to \mathbb{R}$. Furthermore, we say that a set $G \subset \mathrm{GL}(2,\mathbb R)$ is \emph{strongly irreducible} if there is no finite family $V_1,\ldots,V_k$ of proper $1$-dimensional vector subspaces of $\mathbb{R}^2$  such that $A(V_1\cup\cdots\cup V_k)=V_1\cup\cdots\cup V_k$ for all $A\in G$. 

For a real valued function $f$, define $f^+=\max \left\lbrace f,0 \right\rbrace$. The following result by Furstenberg and Kesten in \cite{furstenbergkesten} gives an important analogue to the Law of Large Numbers. 

\begin{thm}[Furstenberg-Kesten]\label{furstenbergkesten}\ \\
	Let $\{Y_i\}_{i\ge1}$ be a sequence of i.i.d. $\mathrm{GL}(d,\mathbb R)$-valued random matrices and $S_n= Y_n Y_{n-1} \cdots Y_2 Y_1$. If \ $\mathbb E\left[\log^+\|Y_1\|\right]<\infty$ and $\lambda$ is the Lyapunov exponent defined in \eqref{lyapdefn}, then almost surely we have
	\[\lambda=\lim_{n\to\infty}\frac1n\log\|S_n\|.\]
\end{thm}

For the rest of this paper, we will suppose that $\mu$ is a probability measure on the group $\mathrm{GL}(2,\mathbb R)$ and that the matrices $\{Y_i\}_{i\ge1}$ are distributed according to $\mu$. However, Theorems \ref{Furstenberg-Kifer}, \ref{furstenberg} and \ref{lepageclt} all have statements valid for matrices in $\mathrm{GL}(d,\mathbb R)$ as well. In \cite{Furstenberg-Kifer1983}, Furstenberg and Kifer give an expression for $\lambda$ in terms of $\mu$ and the $\mu$-invariant probability measures $\nu$ on $\mathbb{P}^{1}\left(\mathbb{R}\right)$. The following result is given in \cite[Theorem 2.2]{Furstenberg-Kifer1983}.

\begin{thm}[Furstenberg-Kifer]\label{Furstenberg-Kifer}\ \\ Let $\mu$ be a probability measure on the group $\mathrm{GL}(2,\mathbb R)$ and $\{Y_i\}_{i\ge1}$ be a sequence of i.i.d. random matrices distributed according to $\mu$. If $\mathbb E \big[ \log^+\|Y_1\|+\log^+\|Y_1^{-1}\| \big]<\infty$, then the Lyapunov exponent is given by 

\[\lambda= \sup_{\nu} \int_{\mathbb{P}^{1}\left(\mathbb{R}\right)}\int_{\mathrm{GL}(2,\mathbb R)}\log\frac{\|A\mathbf{x}\|}{\|\mathbf{x}\|}\,\mathrm d\mu(A)\,\mathrm d\nu(\bar{\mathbf{x}}),\]
where the supremum is taken over all probability measures $\nu $ on $\mathbb{P}^{1}\left(\mathbb{R}\right)$ that are $\mu$-invariant.

\end{thm}

If $\nu$ is the unique $\mu$-invariant probability measure on $\mathbb{P}^{1}\left(\mathbb{R}\right)$, then Theorem \ref{Furstenberg-Kifer} implies that the Lyapunov exponent can be written as \[\lambda=  \int_{\mathbb{P}^{1}\left(\mathbb{R}\right)}\int_{\mathrm{GL}(2,\mathbb R)}\log\frac{\|A \mathbf x \|}{\|\mathbf x \|}\,\mathrm d\mu(A)\,\mathrm d\nu(\bar{\mathbf x}).\]
Sufficient conditions for the existence of such a unique $\nu$ were given by Furstenberg and can be found in \cite[Theorem II.4.1]{bougerol}.

\begin{thm}[Furstenberg]\label{furstenberg}\ \\
	Let $\mu$ be a probability measure on the group $\mathrm{GL}(2,\mathbb R)$ and $\{Y_i\}_{i\ge1}$ be a sequence of i.i.d. random matrices distributed according to $\mu$. Additionally, let $G_\mu$ be the smallest closed subgroup containing the support of $\mu$. Suppose the following hold:
	\begin{enumerate}[(i)]
		\item $\mathbb E \big[\log^+\|Y_1\| \big]<\infty$,	
		\item For $M$ in $G_\mu$, $|\det M|=1$,
		\item $G_\mu$ is not compact, 
		\item $G_\mu$ is strongly irreducible.
	\end{enumerate}
Then there exists a unique $\mu$-invariant probability measure $\nu$ on $\mathbb{P}^{1}\left(\mathbb{R}\right)$ and $\lambda>0$. Moreover, $\nu$ is atomless. Consequently,
	\[\lambda=\int_{-\infty}^\infty\int_{\mathrm{GL}(2,\mathbb R)}\log\frac{\|A\mathbf x\|}{\|\mathbf x \|}\,\mathrm d\mu(A)\,\mathrm d\nu(\bar{\mathbf x}).\]
\end{thm}

Let $A=\begin{pmatrix}a&b\\c&d\end{pmatrix}$ be a $\mathrm{GL}(2,\mathbb R)$-valued random matrix. 
In this paper, we only study matrices $A$ with entry $a$ random and all other entries constant. Let us suppose that the distribution of $a$ is chosen such that the hypotheses of Theorem \ref{furstenberg} hold. Then by a simple computation \cite[pp. 3421]{mtwInvariant} we have that
\[
\lambda=  \int_{-\infty}^\infty\log|cx+d|\,\mathrm d\nu(x),
\] where $\nu$ is the unique $\mu$-invariant probability measure on $\mathbb{P}^{1}\left(\mathbb{R}\right)$. Hence, if $X$ is a random variable distributed according to $\nu$, then
\begin{equation}\label{expectation}
	\lambda=\mathbb E\big[\log|cX+d|\big].
\end{equation}
Moreover, if $A$ and $X$ are independent, we can also conclude that $A\cdot X$ has the same distribution as $X$, which we write as $A\cdot X\sim X$. This follows from the definition of $\mu$-invariance. Thus, a random variable $X$ with law given by the unique $\mu$-invariant distribution on $\mathbb{P}^{1}\left(\mathbb{R}\right)$ must satisfy

\begin{equation}\label{Invariant-Formula}
X\sim\frac{aX+b}{cX+d},
\end{equation}
where $a$ and $X$ are independent. Likewise, the law of any $\mathbb{P}^{1}\left(\mathbb{R}\right)$-valued random variable $X$ which satisfies \eqref{Invariant-Formula} is $\mu$-invariant hence it must be $\nu$. We make use of this distributional identity for the $\mu$-invariant distribution in later sections.

The following result by Le Page can be found in \cite[Theorem V.5.4]{bougerol} and gives a less-studied analogue to the Central Limit Theorem. 

\begin{thm}[Le Page]\label{lepageclt}\ \\
Define $\ell(M)=\max\{\log^+\|M\|,\log^+\|M^{-1}\|\}$ for $M\in\mathrm{GL}(2,\mathbb R)$. Let $\mu$ be a probability measure on the group $\mathrm{GL}(2,\mathbb R)$ and $\{Y_i\}_{i\ge1}$ be a sequence of i.i.d. random matrices distributed according to $\mu$. Moreover, let $G_\mu$ be the smallest closed subgroup containing the support of $\mu$. Suppose the following hold:
\begin{enumerate}[(i)]
	\item $\mathbb{E} \left[\exp \left(t~\ell (Y_1) \right) \right]<\infty$ for some $t>0$,
	\item $G_\mu$ is strongly irreducible,
	\item $\{|\det M|^{-1/2}M:M\in G_\mu\}$ is not contained in a compact subgroup of $\mathrm{GL}(2,\mathbb R).$
\end{enumerate}
Then there exists $\sigma>0$ such that for any $\mathbf{x}\in\mathbb R^2\setminus\{\mathbf{0}\}$, 
\[\frac1{\sqrt n}\left(\log\|S_n\mathbf{x}\|-n\lambda\right)\textrm{ and }\frac1{\sqrt n}\left(\log\|S_n\|-n\lambda\right)\]
converge weakly as $n\to\infty$ to a Gaussian random variable with mean $0$ and variance $\sigma^2$.
\end{thm}

We remark that the relatively recent paper \cite{GL_CLT} has relaxed the exponential moment condition $(i)$ to a second moment condition which cannot be improved. In Section \ref{variance}, we use Monte Carlo simulations to approximate the value of $\sigma^2$ for two matrix models that satisfy the hypotheses of Theorem \ref{lepageclt}. 

\section{$\Bernoulli \left(p\right)$ Parameter Model}	\label{bernp}

In this section we consider a random matrix model where the random entry follows a $\Bernoulli\left(p\right)$ distribution and the parameter of interest is $p$. Recall that a random variable $\epsilon\sim \Bernoulli\left(p\right)$ if ${\mathbb{P}\left(\epsilon=1\right)=p}$ and ${\mathbb{P}\left(\epsilon=0\right)=1-p}$. Let $\mu_p$ be the probability measure on $\mathrm{GL}(2,\mathbb R)$ given by 
 
\begin{equation}\label{BernoullypMatrix1}
\left(\begin{array}{cc}
\epsilon_p & 1\\
1 & 0
\end{array}\right), \ \epsilon_p\sim\Bernoulli\left(p\right),\ 0<p<1.
\end{equation}
It is straightforward to verify that $\mu_p$ satisfies hypotheses $(i)$-$(iv)$ of Theorem \ref{furstenberg}. We verify them here for completeness. For $(i)$, we see that $\mathbb{E}\left[\log^{+}\left\Vert Y_{1}\right\Vert \right]<\infty$
since $\epsilon_{p}$ has finite support. For $(ii)$, consider the subgroup $G$ generated by the possible realizations of \eqref{BernoullypMatrix1}. Since the determinant of each realization has absolute value $1$, so to does every matrix in $G$. Clearly, the closure of $G$, call it $\bar{G}$, is a closed subgroup that contains the support of $\mu_p$. Hence $G_{\mu_p}\subset\bar{G}$. Moreover, since the absolute value of the determinant is continuous, every matrix in $\bar{G}$ also has determinant with absolute value $1$. It follows that the same holds for $G_{\mu_p}$ as required. 

For $(iii)$, we first let $F_{0},F_{1},F_{2},F_{3},\dots$ be the usual Fibonacci
sequence $0,1,1,2,3,5,\dots$ Then
a simple calculation shows that for each positive integer $n$, we
have 
\[
\left(\begin{array}{cc}
1 & 1\\
1 & 0
\end{array}\right)^{n}=\left(\begin{array}{cc}
F_{n+1} & F_{n}\\
F_{n} & F_{n-1}
\end{array}\right).
\]
Since the powers of the matrix \eqref{BernoullypMatrix1} with $\epsilon_{p}=1$ must
be in $G_{\mu_p}$ and the norm of the powers grow arbitrarily large with
large $n$, it follows that $G_{\mu_p}$ is unbounded and hence not compact. 

Lastly, hypothesis $(iv)$ can be checked by way of an equivalent condition given in \cite[Proposition II.4.3]{bougerol}. This condition is met as long as for any $\bar{\mathbf{x}}\in\mathbb{P}^{1}\left(\mathbb{R}\right)$,
the set $S_{\bar{\mathbf{x}}}=\left\{ M\cdot\bar{\mathbf{x}}: M\in G_{\mu_p}\right\} $
has more than two elements. To see that this holds, suppose at least one of $x_1,x_2\in\mathbb{R}$ is nonzero and consider $\mathbf{x}=\left(\begin{array}{c}
x_1\\
x_2
\end{array}\right)$. Drawing the matrix $M$ from \eqref{BernoullypMatrix1} with $\epsilon_{p}=1$, we have 
\begin{align*}
M\cdot\bar{\mathbf{x}} & =\left(\begin{array}{cc}
1 & 1\\
1 & 0
\end{array}\right)\cdot\bar{\mathbf{x}}=\overline{\left(\begin{array}{c}
x_1+x_2\\
x_1
\end{array}\right)}=1+\frac{x_2}{x_1}\in S_{\bar{\mathbf{x}}},\\
M^2\cdot\bar{\mathbf{x}} 
&=
\left(\begin{array}{cc}
2 & 1\\
1 & 1
\end{array}\right)\cdot\bar{\mathbf{x}}=
\overline{\left(\begin{array}{c}
2x_1+x_2\\
x_1+x_2
\end{array}\right)}=1+\frac{x_1}{x_1+x_2}\in S_{\bar{\mathbf{x}}},\\
M^3\cdot\bar{\mathbf{x}}&=
\left(\begin{array}{cc}
3 & 2\\
2 & 1
\end{array}\right)\cdot\bar{\mathbf{x}} =
\overline{\left(\begin{array}{c}
3x_1+2x_2\\
2x_1+x_2
\end{array}\right)}=1+\frac{x_1+x_2}{2x_1+x_2}\in S_{\bar{\mathbf{x}}}.
\end{align*}
Since for any $\mathbf{x}$, each of these elements in $S_{\bar{\mathbf{x}}}$ is distinct, it follows that hypothesis $(iv)$ holds. 

Since $\mu_p$ satisfies hypotheses $(i)$-$(iv)$ of Theorem \ref{furstenberg}, we know there exists a unique $\mu_p$-invariant distribution $\nu_p$ that satisfies \eqref{eq:invariance} and that $\nu_p$ is atomless. Then by \eqref{Invariant-Formula}, any random variable $X_p$ with law $\nu_p$ must satisfy the distributional identity 
\begin{equation}\label{Bernp2}
X_p\sim\frac{1}{X_p}+\epsilon_p,
\end{equation}
where $\epsilon_p\sim \Bernoulli\left(p\right)$ and is independent
of $X_p$. Likewise, the law of any $\mathbb{P}^{1}\left(\mathbb{R}\right)$-valued random variable $X_p$ which satisfies \eqref{Bernp2} is $\mu_p$-invariant hence it must be $\nu_p$. Using \eqref{Bernp2} and the fact that $\nu_p$ is atomless, it is not hard to see that $X_p\in (0,\infty)$ almost surely. See Goswami \cite{Goswami2004} for this fact and other facts about $X$, including an expression for its cumulative distribution function in terms of a continued fraction expansion. In Figures \ref{fig:histogram} and \ref{fig:CDF} we show the empirical distribution of $100\,000$ independent draws from $\nu_{1/2}$ and remark that the fractal nature of this probability measure is clearly apparent.

\begin{figure}[H]
\centering
\begin{subfigure}{.5\textwidth}
  \centering
\includegraphics[scale=0.5]{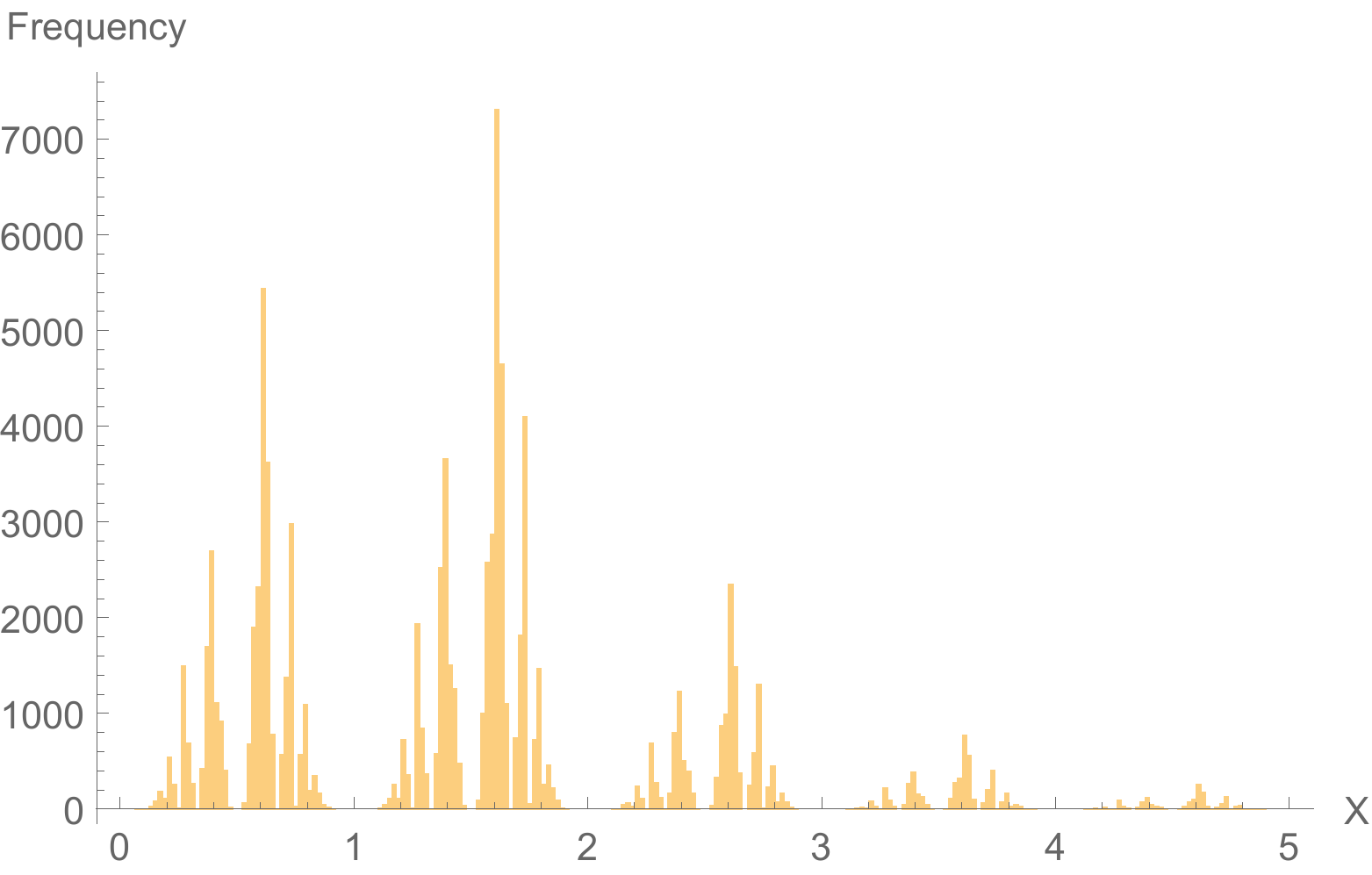}
\caption{Histogram}
\label{fig:histogram}
\end{subfigure}%
\begin{subfigure}{.5\textwidth}
\centering
\includegraphics[scale=0.5]{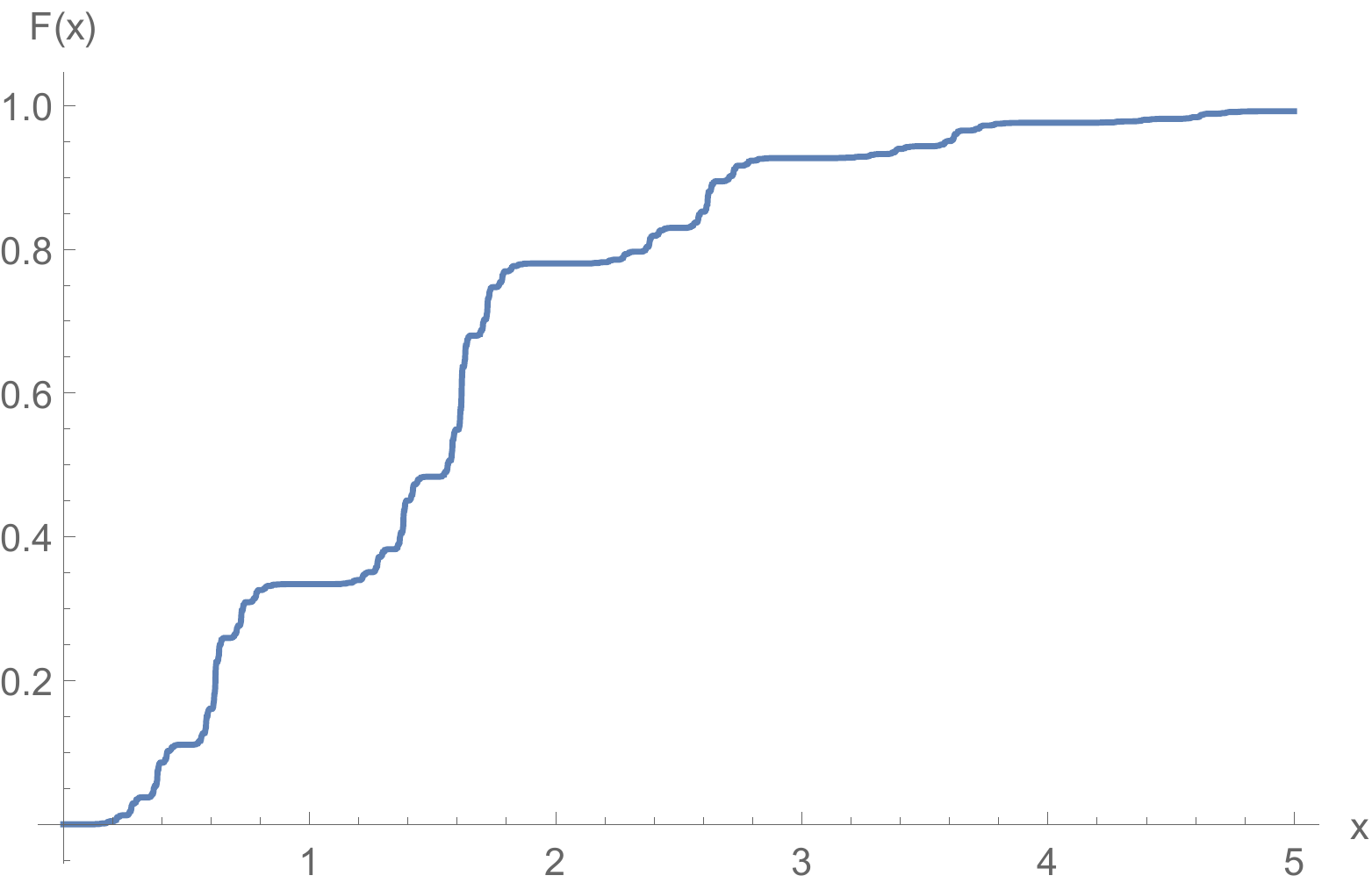}
\caption{CDF}
\label{fig:CDF}
\end{subfigure}
\caption{}
\end{figure}

Let $\lambda(p)$ be the Lyapunov exponent related to $\mu_p$. Using \eqref{expectation} and the fact that $X_p$ is non-negative, we can write the Lyapunov exponent associated with $\mu$ as 

\begin{equation}\label{eq:lambda_p}
\lambda(p)=\mathbb{E}\left[\log X_p\right].
\end{equation}

\subsection{The general $0<p<1$ case}\label{bernp1}

In this subsection we study $\lambda(p)$ for general $0<p<1$ and obtain two sided bounds depending on the parameter $p$. First we prove some identities for $\mathbb{E}\left[\log X_{p}\right]$. We begin by establishing an identity for $\mathbb{E}\left[\log X_p\right]$ which will be later generalized for the $p=\frac{1}{2}$ case and used in proving a limiting result.  

\begin{lem}\label{prop:p_identity}
If $X_p$ is a $\mathbb{P}^{1}\left(\mathbb{R}\right)$-valued random variable satisfying \eqref{Bernp2}, then
\[
0<\mathbb{E}\left[\log X_p\right]<\infty 
\] 
and
\[
\mathbb{E}\left[\log X_{p}\right]=\frac{p}{3}\mathbb{E}\left[\log\left(2X_{p}+1\right)\right].
\]
\end{lem}

\begin{proof}
Let $X_p$ be a random variable satisfying \eqref{Bernp2}. Then $X_p$ has law $\nu_p$ given by Theorem \ref{furstenberg} applied to random matrices of the form \eqref{BernoullypMatrix1}. Consequently, we have that $0<\lambda(p)<\infty$ and it follows from \eqref{eq:lambda_p} that $\mathbb{E}\left[\log X_p\right]$ is positive and finite. Using \eqref{Bernp2}, we start by writing
\begin{align}
\mathbb{E}\left[\log X_p\right] & =\mathbb{E}\left[\log\left(\frac{1}{X_p}+\epsilon\right)\right]\nonumber \\
 & =(1-p)\mathbb{E}\left[\log\left(\frac{1}{X_p}\right)\right]+p\mathbb{E}\left[\log\left(\frac{1}{X_p}+1\right)\right]\nonumber \\
 & =-(1-p)\mathbb{E}\left[\log X_p\right]+p\mathbb{E}\left[\log\left(\frac{1+X_p}{X_p}\right)\right]\nonumber \\
 & =-\mathbb{E}\left[\log X_p\right]+p\mathbb{E}\left[\log\left(1+X_p\right)\right].\label{eq:7}
\end{align}
Adding $\mathbb{E}\left[\log X_p\right]$ to both sides of \eqref{eq:7}
and dividing by 2 results in
\begin{equation}
\mathbb{E}\left[\log X_p\right]=\frac{p}{2}\mathbb{E}\left[\log\left(1+X_p\right)\right].\label{eq:8}
\end{equation}
Continuing in a similar fashion with \eqref{eq:8}, we obtain
\begin{align}
\mathbb{E}\left[\log X_p\right] & =\frac{p}{2}\mathbb{E}\left[\log\left(1+\frac{1}{X_p}+\epsilon\right)\right]\nonumber \\
 & =\frac{p(1-p)}{2}\mathbb{E}\left[\log\left(1+\frac{1}{X_p}\right)\right]+\frac{p^2}{2}\mathbb{E}\left[\log\left(2+\frac{1}{X_p}\right)\right]\nonumber \\
 & =\frac{p(1-p)}{2}\mathbb{E}\left[\log\left(\frac{X_p+1}{X_p}\right)\right]+\frac{p^2}{2}\mathbb{E}\left[\log\left(\frac{2X_p+1}{X_p}\right)\right]\nonumber \\
 & =\frac{p(1-p)}{2}\mathbb{E}\left[\log\left(X_p+1\right)\right]+\frac{p^2}{2}\mathbb{E}\left[\log\left(2X_p+1\right)\right]-\frac{p}{2}\mathbb{E}\left[\log X_p\right]\nonumber \\
 & =\left(1-\frac{3p}{2}\right)\mathbb{E}\left[\log X\right]+\frac{p^2}{2}\mathbb{E}\left[\log\left(2X+1\right)\right],\label{eq:9}
\end{align}
where we use \eqref{eq:8} in the last equality. Subtracting $\left(1-\frac{3p}{2}\right)\mathbb{E}\left[\log X\right]$ from both
sides of \eqref{eq:9} leads to
\begin{equation*}
\mathbb{E}\left[\log X\right]=\frac{p}{3}\mathbb{E}\left[\log\left(2X+1\right)\right],
\end{equation*}
completing the proof. 
\end{proof}

\begin{lem}\label{Expectation-Identity1}
If $X_p$ is a $\mathbb{P}^{1}\left(\mathbb{R}\right)$-valued random variable satisfying \eqref{Bernp2}, then
\begin{equation}
\mathbb{E}\left[\log\left(X_{p}\right)\cdot\mathbf{1}_{\left(X_{p}<1\right)}\right]=\left(p-1\right)\mathbb{E}\left[\log\left(X_{p}\right)\cdot\mathbf{1}_{\left(X_{p}>1\right)}\right],\label{eq:c}
\end{equation}

\begin{equation}
\mathbb{E}\left[\log \left(X_{p}\right)\cdot\mathbf{1}_{\left(X_{p}>1\right)}\right]=\frac{1}{p}\mathbb{E}\left[\log X_{p}\right],\label{eq:a}
\end{equation} 
and
\begin{equation}
\mathbb{E}\left[\log\left(X_{p}\right)\cdot\mathbf{1}_{\left(X_{p}<1\right)}\right]=\frac{p-1}{p}\mathbb{E}\left[\log X_{p}\right].\label{eq:b}
\end{equation}

\end{lem}

\begin{proof}
Recalling that the distribution of $X_p$ has non-negative support, observe that
\begin{eqnarray*}
\mathbb{E}\left[\log\left(X_{p}\right)\cdot\mathbf{1}_{\left(X_{p}<1\right)}\right] & = & p~\mathbb{E}\left[\log\left(\frac{1}{X_{p}}+1\right)\cdot \mathbf{1}_{\left(\frac{1}{X_{p}}+1<1\right)}\right]+\left(1-p\right)\mathbb{E}\left[\log\left(\frac{1}{X_{p}}\right)\cdot \mathbf{1}_{\left(\frac{1}{X_{p}}<1\right)}\right]\\
 & = & 0+\left(1-p\right)\mathbb{E}\left[\log\left(\frac{1}{X_{p}}\right)\cdot \mathbf{1}_{\left(X_{p}>1\right)}\right]\\
 & = & \left(p-1\right)\mathbb{E}\left[\log\left(X_{p}\right)\cdot \mathbf{1}_{\left(X_{p}>1\right)}\right].
\end{eqnarray*}
This proves \eqref{eq:c} which, along with the fact that the distribution of $X_p$ is atomless, allows us to write
\begin{eqnarray*}
\mathbb{E}\left[\log X_{p}\right] & = & \mathbb{E}\left[\log\left(X_{p}\right)\cdot \mathbf{1}_{\left(X_{p}>1\right)}\right]+\mathbb{E}\left[\log\left(X_{p}\right)\cdot \mathbf{1}_{\left(X_{p}<1\right)}\right]\\
 & = & \mathbb{E}\left[\log\left(X_{p}\right)\cdot \mathbf{1}_{\left(X_{p}>1\right)}\right]+\left(p-1\right)\mathbb{E}\left[\log\left(X_{p}\right)\cdot \mathbf{1}_{\left(X_{p}>1\right)}\right]\\
 & = & p~\mathbb{E}\left[\log\left(X_{p}\right)\cdot \mathbf{1}_{\left(X_{p}>1\right)}\right]
\end{eqnarray*}
which proves \eqref{eq:a}. Combining these two identities now leads to \eqref{eq:b}.
\end{proof}

Next we use these results to establish bounds on the Lyapunov exponent which are dependent on $p$.

\begin{thm}\label{UBprop}
Let $\mu_p$ be the probability measure on $\mathrm{GL}(2,\mathbb R)$ given by \eqref{BernoullypMatrix1}. Then the Lyapunov exponent $\lambda(p)$ associated with $\mu_p$ can be estimated by 
\[
\frac{p\log 3}{4-p}\leq\lambda(p)\leq\frac{p\log 3}{2}.
\]

\end{thm}

\begin{proof}
Beginning with the upper estimate, first note that $\log(2x+1)\leq \log(3x)$ for $x\geq 1$. Now using Lemma \ref{prop:p_identity} and \eqref{eq:a}, we can write 
\begin{eqnarray}
\mathbb{E}\left[\log X_{p}\right] & = & \frac{p}{3}\mathbb{E}\left[\log(2X_{p}+1)\right]\nonumber \\
 & = & \frac{p}{3}\mathbb{E}\left[\log(2X_{p}+1)\cdot \mathbf{1}_{\left(X_{p}<1\right)}+\log(2X_{p}+1)\cdot \mathbf{1}_{\left(X_{p}>1\right)}\right]\nonumber \\
 & \leq & \frac{p}{3}\Big(\log 3~\mathbb{P}\left(X_{p}<1\right)+\mathbb{E}\left[\log(3X_{p})\cdot \mathbf{1}_{\left(X_{p}>1\right)}\right]\Big)\nonumber \\
 & = & \frac{p}{3}\Big(\log3+\mathbb{E}\left[\log(X_{p})\cdot \mathbf{1}_{\left(X_{p}>1\right)}\right]\Big)\nonumber \\
 & = & \frac{p}{3}\log3+\frac{1}{3}\mathbb{E}\left[\log X_{p}\right].\label{eq:g}
\end{eqnarray}
Subtracting $\frac{1}{3}\mathbb{E}\left[\log X_{p}\right]$
from both sides of \eqref{eq:g} and recalling \eqref{eq:lambda_p} leads to the desired result. 

For the lower estimate, we proceed similarly, noting that $\log(2x+1)\geq\log(3x)$ for $0<x\leq 1$ and using \eqref{eq:b} instead of \eqref{eq:a} to write
\begin{eqnarray}
\mathbb{E}\left[\log X_{p}\right] & = & \frac{p}{3}\mathbb{E}\left[\log(2X_{p}+1)\cdot \mathbf{1}_{\left(X_{p}<1\right)}+\log(2X_{p}+1)\cdot \mathbf{1}_{\left(X_{p}>1\right)}\right]\nonumber \\
 & \geq & \frac{p}{3}\Big(\mathbb{E}\left[\log(3X_{p})\cdot \mathbf{1}_{\left(X_{p}<1\right)}\right]+\log3~\mathbb{P}\left(X_{p}>1\right)\Big)\nonumber \\
 & = & \frac{p}{3}\Big(\log3+\mathbb{E}\left[\log(X_{p})\cdot \mathbf{1}_{\left(X_{p}<1\right)}\right]\Big)\nonumber \\
 & = & \frac{p}{3}\log3+\frac{p-1}{3}\mathbb{E}\left[\log X_{p}\right].\label{eq:h}
\end{eqnarray}
Now the lower bound follows from a simple rearrangement of \eqref{eq:h}.
\end{proof}

\subsection{Approximating $\lambda(p)$ by simulation}\label{sec:approx}\hfill\\

Let $\{Y_i\}_{i\ge1}$ be an i.i.d. sequence drawn from $\mu_p$, and for some $\mathbf{x}\in\mathbb{R}^2$ with $\|\mathbf{x}\|=1$, construct $\{U_i\}_{i\geq 0}$ recursively by $U_0=\mathbf{x}$ and $U_i=Y_i\frac{U_{i-1}}{\|U_{i-1}\|}$. Now, with $S_n=Y_n Y_{n-1}\cdots Y_2 Y_1$ and $S_0=Y_0=I$, we have
\begin{align}
\frac{1}{n}\log\|S_n\mathbf{x}\|&=\frac{1}{n}\sum_{i=1}^n\log\frac{\|S_i\mathbf{x}\|}{\|S_{i-1}\mathbf{x}\|}\nonumber \\
&=\frac{1}{n}\sum_{i=1}^n\log\left\|Y_i\frac{Y_{i-1}\dots Y_1\mathbf{x}}{\|Y_{i-1}\dots Y_1\mathbf{x}\|}\right\|\nonumber \\
&=\frac{1}{n}\sum_{i=1}^n\log\left\|U_i\right\|.\label{eq:lln_approx}
\end{align}
Hence it follows from Theorem \ref{furstenbergkesten} that we can approximate $\lambda$ by the right-hand side of \eqref{eq:lln_approx} with $n$ large. Since the $\log\left\|U_i\right\|$ terms aren't growing with $i$, this avoids numerical overflow issues and makes for a robust Monte Carlo scheme.

In Figure \ref{fig:BernPLambdaBounds}, we plot simulations for $\lambda(p)$ in black and the upper and lower bounds from Theorem \ref{UBprop} in blue. We discretize $[0,1]$ into sub-intervals of length $0.01$ and use $n=1\,000\,000$ in the Monte Carlo scheme described above.

\begin{figure}[H]
\centering
\includegraphics[scale=0.5]{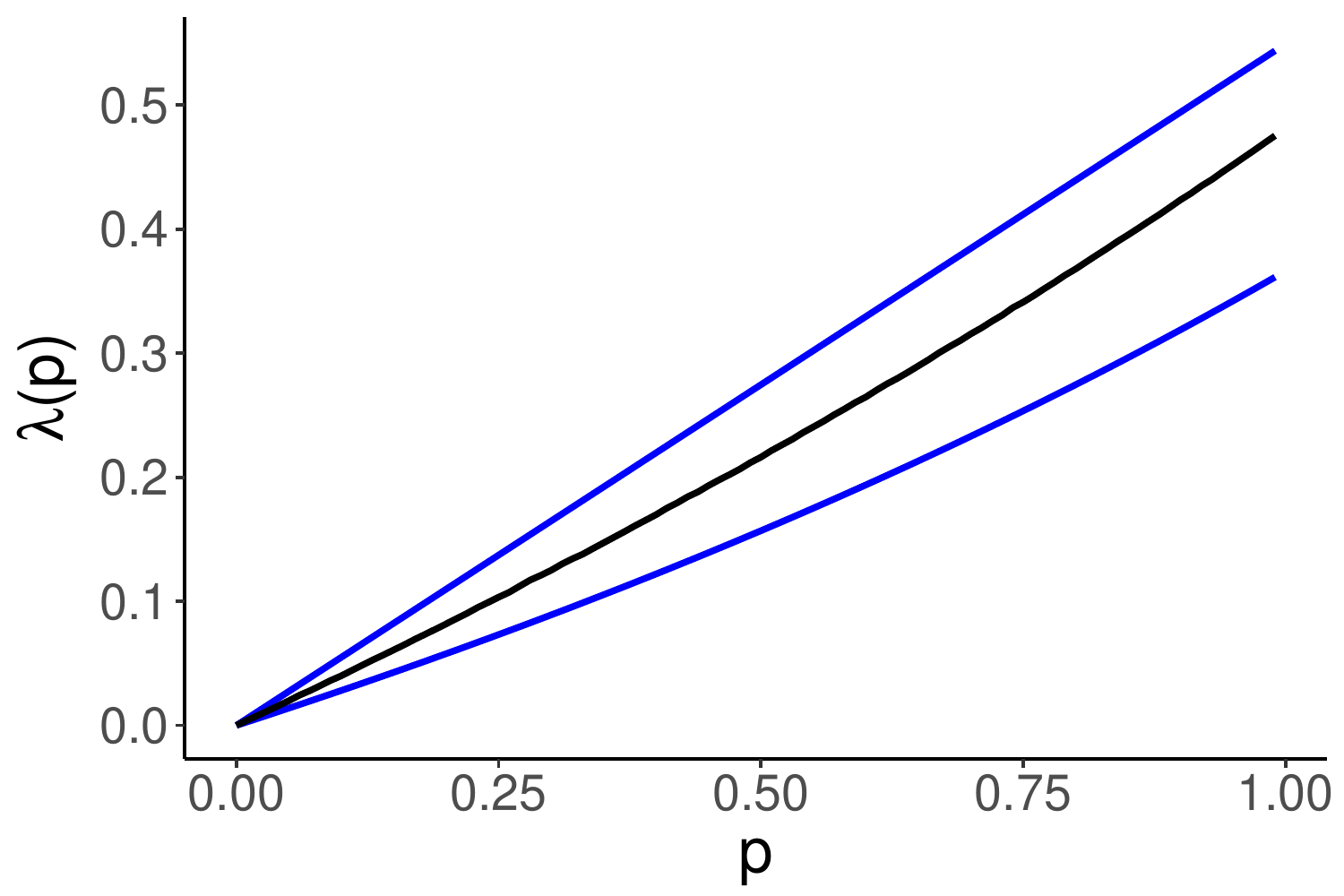}
\caption{$n=1\,000\,000$}
\label{fig:BernPLambdaBounds}
\end{figure}

\subsection{The $p=\frac{1}{2}$ case}\label{bernhalf}

In this section we study $\lambda :=\lambda\left(\frac{1}{2}\right)$ in more detail. To set notation, recall that a random variable $\epsilon\sim \Bernoulli\left(\frac{1}{2}\right)$ if $\mathbb{P}\left(\epsilon=1\right)=\mathbb{P}\left(\epsilon=0\right)=\frac{1}{2}$. The probability measure $\mu$ on $\mathrm{GL}(2,\mathbb R)$ that we consider is given by 
\begin{equation}
\left(\begin{array}{cc}
\epsilon & 1\\
1 & 0
\end{array}\right),\ \epsilon\sim \Bernoulli\left(\frac{1}{2}\right).\,\,\,\,\,\label{Ber(1/2) - Intro}
\end{equation}
We know by the general $p$ case that there exists a unique $\mu$-invariant distribution $\nu$ that satisfies \eqref{eq:invariance} and that $\nu$ is atomless. Then by \eqref{Invariant-Formula}, any random variable $X$ with law $\nu$ must satisfy the distributional identity 
\begin{equation}\label{Invariant1}
X\sim\frac{1}{X}+\epsilon,
\end{equation}
where $\epsilon\sim \Bernoulli\left(\frac{1}{2}\right)$ and is independent
of $X$.
Using \eqref{expectation} and the fact that $X$ is non-negative, we can write the Lyapunov exponent associated with $\mu$ as  
\begin{equation}\label{eq:lambda_log}
\lambda=\mathbb{E}\left[\log X\right].
\end{equation}

Unlike in the general case, we will be able to obtain a sequence of upper and lower bounds that converge to $\lambda$. Recall that by Lemma \ref{prop:p_identity} for $p=\frac{1}{2}$ we showed that \[
0<\mathbb{E}\left[\log X\right]<\infty 
\] 
and
\begin{equation}\label{FirstIdentity}
\mathbb{E}\left[\log X\right]=\frac{1}{6}\mathbb{E}\left[\log\left(2X+1\right)\right].
\end{equation}

We will prove a string of identities akin to equation \eqref{FirstIdentity} in a similar fashion. Here we list a few examples.  
\begin{equation}\label{Expectation of Poly}
\begin{aligned}
\mathbb{E}\left[\log X\right] & =\frac{1}{6}\mathbb{E}\left[\log\left(2X+1\right)\right] \\
 & =\frac{1}{14}\mathbb{E}\left[\log\left(3X+2\right)(X+2)\right] \\
 & =\frac{1}{32}\mathbb{E}\left[\log\left(5X+3\right)\left(3X+1\right)\left(2X+3\right)\left(2X+1\right)\right] \\
 & =\frac{1}{72}\mathbb{E}\left[\log\left(8X+5\right)\left(4X+3\right)\left(5X+2\right)\left(3X+2\right)\left(3X+5\right)\left(X+3\right)\left(3X+2\right)\left(X+2\right)\right]\\
 & \ \ \vdots
\end{aligned}
\end{equation}
 The string of identities above is obtained by iteratively exploiting the distributional equivalence of $X$ and $\frac{1}{X}+\epsilon$, the independence of $X$ and $\epsilon$, and elementary logarithmic identities. We will later see that an interesting pattern emerges. At the first step of the iteration, we are looking at the expected value of the $\log$ of one affine function of $X$ that is obtained by taking the inner product of the vector $(2,1)$ and the vector $(X,1)$. As we move to the second step of the iteration, we encounter the expectation of the $\log$ of the product of two affine functions of $X$. The first one is obtained by taking the inner product of $(3,2)$ and $(X,1)$, while the second is obtained by taking the inner product of $(1,2)$ and $(X,1)$. At the third step, we encounter the expected value of the $\log$ of the product of four $\left(=2^{3-1}\right)$ affine functions of $X$; these are obtained by respectively taking the inner product of $(X,1)$ with the vectors $(5,3)$, $(3,1)$, $(2,3)$, and $(2,1)$. 

In what follows, we represent the vectors generating the aforesaid affine functions of $X$ via inner products with $(X,1)$, which we call ``coefficient pairs'', in an array where the row number corresponding to the $n^{\text{th}}$ step of the iteration is $n-1$. The first four rows of the array are shown below. We use the symbol  $\mapsto$  to map the collection of coefficient pairs to the real number representing the product of the sum of entries in each coefficient pair in the row; we make extensive use of these quantities later on. 
\begin{equation}\label{an,bn array}
\begin{aligned}
n=0 & \hspace{6mm}\left(2,1\right) \mapsto 3 \\
n=1 & \hspace{6mm}\left(3,2\right)\left(1,2\right)  \mapsto 5\cdot3 = 15\\
n=2 & \hspace{6mm}\left(5,3\right)\left(3,1\right)\left(2,3\right)\left(2,1\right) \mapsto 8\cdot4\cdot5\cdot3=480 \\
n=3 & \hspace{6mm}\left(8,5\right)\left(4,3\right)\left(5,2\right)\left(3,2\right)\left(3,5\right)\left(1,3\right)\left(3,2\right)\left(1,2\right)  \mapsto 13\cdot7\cdot7\cdot5\cdot8\cdot4\cdot5\cdot3=1528800 \\
\vdots\hspace{3.2mm} & \hspace{75.5mm}\vdots
\end{aligned}
\end{equation}
For the $k^\text{th}$ coefficient pair in row $n$, let $a_{n}^{k}$ denote the first element and $b_{n}^{k}$ the second. 
To illustrate this notational convention, consider the example
$\frac{1}{14}\mathbb{E}\left[\log\left(3X+2\right)\left(X+2\right)\right]$
from \eqref{Expectation of Poly}. This is in row $n=1$, so we would refer to the 3 in $\left(3X+2\right)$
as $a_{1}^{1}$ and the 2 as $b_{1}^{1}$. Similarly, the coefficient of $X$ in $\left(X+2\right)$ would be labeled $a_{1}^{2}$ and the
$2$ would be labeled $b_{1}^{2}$. In terms of $a^k_n$ and $b^k_n$, the expression is $\frac{1}{14}\mathbb{E}\left[\log\left(a_{1}^{1}X+b_{1}^{1}\right)\left(a_{1}^{2}X+b_{1}^{2}\right)\right]$. Now we can define the multi-level recursion that describes the array given in \eqref{an,bn array}.
\begin{defn}\label{def:recursion}
Set $a_{0}^{1}=2$ and $b_{0}^{1}=1$. For any $n\in \mathbb{Z}_{\geq 0}$, define 
\[
\begin{array}{ccl}
\left(a_{n+1}^{k},b_{n+1}^{k}\right):=\left(a_{n}^{k}+b_{n}^{k},a_{n}^{k}\right),\hfill \hfill & \mbox{ for }k=1,\dots,2^{n},\hfill\\
\left(a_{n+1}^{k},b_{n+1}^{k}\right):=\left(b_{n}^{k-2^{n}},a_{n}^{k-2^{n}}\right),\hfill \hfill& \mbox{ for }k=2^{n}+1,\dots,2^{n+1}.\hfill\\
\\
\end{array}
\]
\end{defn}

We observe several conspicuous patterns in \eqref{an,bn array} which are implicit in Definition \ref{def:recursion}. For instance, row $n$ is made up of $2^{n}$ pairs and the second half of row $n$ is simply row $n-1$ where the elements within the coefficient pairs have been switched. One property that will prove useful is the fact that the first coefficient pair in each row dominates the other pairs occurring in that row in the sense that
\begin{equation}\label{eq:dominate}
a_n^1\geq a_n^k~\text{ and }~b_n^1\geq b_n^k~\text{ for all }~1\leq k\leq 2^n.
\end{equation}
This follows from the recursion in Definition \ref{def:recursion} and induction on $n$.

To exhibit a less obvious pattern, we first recall that a ``Fibonacci-like sequence'' of numbers $f_0,f_1,f_2\dots$ is a sequence determined by the initial values $f_0,f_1$ such that \[
f_{n+1}=f_{n}+f_{n-1}
\]
for all $n\in \mathbb{N}$. When $f_0=0,f_1=1$, we recover the standard Fibonacci sequence. Fibonacci-like sequences can be expressed by an explicit formula. Let
$f_{n}(f_0,f_1)$ represent the $n$th term in the sequence given initial values $f_0,f_1$. If 
\[
\phi_{1}=\frac{1+\sqrt{5}}{2}\text{ and }\phi_{2}=\frac{1-\sqrt{5}}{2},
\]
then 
\begin{equation}
f_{n}(f_{0},f_{1})=\frac{f_{1}-f_{0}\phi_{2}}{\sqrt{5}}\left(\phi_{1}\right)^{n}+\frac{f_{0}\phi_{1}-f_{1}}{\sqrt{5}}\left(\phi_{2}\right)^{n}.\label{fibonacci}
\end{equation}
Now note that given $n\in\mathbb{N}$ and $k\in\left\{ 1,\dots,2^{n-1}\right\}$,
we have 
\[
a_{n+1}^{k}=a_{n}^{k}+b_{n}^{k}=a_{n}^{k}+a_{n-1}^{k}
\]
and
\[
b_{n+1}^{k}=a_{n}^{k}=a_{n-1}^k+b_{n-1}^{k}=b_{n}^k+b_{n-1}^{k}.
\]
Thus, for each $k$, the sequences $\{a_{n}^{k}\}$ and $\{b_{n}^{k}\}$ will be Fibonacci-like sequences in $n$ for $n$ large enough. 

We use these observations to help establish bounds on the Lyapunov exponent. In order to find
suitable estimates, we first need to establish some preliminary results. These involve proving the string of identities given in \eqref{Expectation of Poly}. We also need to prove some elementary inequalities involving the logarithm of the polynomials
given inside the expectations in \eqref{Expectation of Poly}. 

First, we extend the identities given in \eqref{Expectation of Poly} to all $n$. 
\begin{lem}\label{Prop:3.2}
If $X$ is a $\mathbb{P}^{1}\left(\mathbb{R}\right)$-valued random variable satisfying \eqref{Invariant1}, then
\begin{equation}
\mathbb{E}\left[\log X\right]=\frac{1}{(n+6)2^{n}}\mathbb{E}\left[\log\left(\prod_{k=1}^{2^{n}}\left(a_{n}^{k}X+b_{n}^{k}\right)\right)\right]\label{eq:9-1}
\end{equation}
 for all $n\in \mathbb{Z}_{\geq 0}$.
\end{lem}

\begin{proof}
We begin with $n=0$. By Lemma \ref{prop:p_identity} with $p=\frac{1}{2}$ we have,
\begin{align*}
\mathbb{E}\left[\log X\right] & =\frac{1}{6}\mathbb{E}\left[\log\left(2X+1\right)\right]\\
& =\frac{1}{(0+6)2^{0}}\mathbb{E}\left[\log\left(a_{0}^{1}X+b_{0}^{1}\right)\right].
\end{align*}
Now suppose \eqref{eq:9-1} holds for $n$. We shall show that \eqref{eq:9-1} holds for $n+1$. Note that
\begin{align}
\mathbb{E}\left[\log X\right] & =\frac{1}{(n+6)2^{n}}\mathbb{E}\left[\log\left(\prod_{k=1}^{2^{n}}\left(a_{n}^{k}X+b_{n}^{k}\right)\right)\right]\nonumber \\
 & =\frac{1}{(n+6)2^{n}}\left(\frac{1}{2}\mathbb{E}\left[\log\left(\prod_{k=1}^{2^{n}}\left(a_{n}^{k}\left(\frac{1}{X}+1\right)+b_{n}^{k}\right)\right)\right]+\frac{1}{2}\mathbb{E}\left[\log\left(\prod_{k=1}^{2^{n}}\left(a_{n}^{k}\left(\frac{1}{X}\right)+b_{n}^{k}\right)\right)\right]\right)\nonumber \\
 & =\frac{1}{(n+6)2^{n+1}}\left(\mathbb{E}\left[\log\left(\prod_{k=1}^{2^{n}}\left(\frac{a_{n}^{k}}{X}+a_{n}^{k}+b_{n}^{k}\right)\right)\right]+\mathbb{E}\left[\log\left(\prod_{k=1}^{2^{n}}\left(\frac{a_{n}^{k}}{X}+b_{n}^{k}\right)\right)\right]\right)\nonumber \\
 & =\frac{1}{(n+6)2^{n+1}}\left(\mathbb{E}\left[\log\left(\prod_{k=1}^{2^{n}}\left(\frac{a_{n}^{k}+\left(a_{n}^{k}+b_{n}^{k}\right)X}{X}\right)\right)\right]+\mathbb{E}\left[\log\left(\prod_{k=1}^{2^{n}}\left(\frac{a_{n}^{k}+b_{n}^{k}X}{X}\right)\right)\right]\right)\nonumber \\
 & =\frac{1}{(n+6)2^{n+1}}\mathbb{E}\left[\log\left(\prod_{k=1}^{2^{n}}\Big(a_{n}^{k}+\left(a_{n}^{k}+b_{n}^{k}\right)X\Big)\prod_{k=1}^{2^{n}}\left(a_{n}^{k}+b_{n}^{k}X\right)\right)\right]-\frac{\mathbb{E}\left[\log X\right]}{(n+6)}.\label{eq:10}
\end{align}
Moving the last term on the right-hand side of \eqref{eq:10} to the left leads to
\begin{align*}
\mathbb{E}\left[\log X\right] & =\frac{1}{\big((n+1)+6\big)2^{n+1}}\mathbb{E}\left[\log\left(\prod_{k=1}^{2^{n+1}}\left(a_{n+1}^{k}X+b_{n+1}^{k}\right)\right)\right].
\end{align*}
Here we have combined and simplified the products appearing in \eqref{eq:10} by using the recursion from Definition \ref{def:recursion}. The result now follows by induction.
\end{proof}

We now prove the elementary inequalities needed to estimate \eqref{eq:9-1}. 
\begin{lem}
Let $n\in \mathbb{Z}_{\geq 0}$. For $x\geq 1$,
\begin{equation}\label{lem1}
\log\left(\prod_{k=1}^{2^{n}}\left(a_{n}^{k}x+b_{n}^{k}\right)\right)\leq\log\left(x^{2^{n}}\prod_{k=1}^{2^{n}}\left(a_{n}^{k}+b_{n}^{k}\right)\right).
\end{equation}
Conversely, when $0<x\leq 1$,
\begin{equation}\label{lem2}
\log\left(\prod_{k=1}^{2^{n}}\left(a_{n}^{k}x+b_{n}^{k}\right)\right)\geq\log\left(x^{2^{n}}\prod_{k=1}^{2^{n}}\left(a_{n}^{k}+b_{n}^{k}\right)\right).
\end{equation}
\end{lem}

\begin{proof}
Note that when $x\geq 1$, we have $a^k_n x+b^k_n\leq x (a^k_n+b^k_n)$. Taking products and the $\log$ of both sides gives us the desired result. The proof of the $0<x\leq 1$ case follows similarly.
\end{proof} 

Using \eqref{lem1} and \eqref{lem2}, we can prove that the Lyapunov exponent
is bounded by terms dependent only on $n$. First, we define the following quantities that appear as the rightmost entries of \eqref{an,bn array}. 

\begin{defn}\label{def:c_n}
For each $n\in \mathbb{Z}_{\geq 0}$, let $c_{n}$ be the product of the sums of coefficient
pairs in row $n$ of \eqref{an,bn array}. That is, 
\[
c_{n}=\prod_{k=1}^{2^{n}}\left(a_{n}^{k}+b_{n}^{k}\right).
\]
\end{defn}
\noindent For example, $c_0,\dots, c_3$ are displayed in \eqref{an,bn array}. We remark that the recursion from Definition \ref{def:recursion} implies
\begin{equation}\label{eq:cn_rec}
c_n=c_{n-1}\prod_{k=1}^{2^{n-1}}\left(a_{n}^{k}+b_{n}^{k}\right)=\prod_{k=1}^{2^{n}}a_{n+1}^{k}.
\end{equation}

Now we can state our main result of this section.
\begin{thm}\label{thm:bounds}
Let $\mu$ be the probability measure on $\mathrm{GL}(2,\mathbb R)$ given by \eqref{Ber(1/2) - Intro}. Then for each $n\in \mathbb{Z}_{\geq 0}$, the Lyapunov exponent $\lambda$ associated with $\mu$ can be estimated by 
\begin{equation}\label{Bounds}
p_n\leq\lambda\leq q_n,
\end{equation}
where 
\begin{equation}\label{ConvergingSequence}
p_n =\frac{\log c_{n}}{\left(n+7\right)2^{n}} \,\,\, \mbox{and}\,\,\, q_n = \frac{\log c_{n}}{\left(n+4\right)2^{n}}.
\end{equation}
Moreover, 
\[
\lim_{n\to \infty}p_{n} = \lim_{n\to\infty} q_{n} = \lambda.
\]
\end{thm}

\begin{proof}
Fix $n\in \mathbb{N}\cup \left\lbrace 0 \right\rbrace$ and let $X$ be a $\mathbb{P}^{1}\left(\mathbb{R}\right)$-valued random variable satisfying \eqref{Invariant1}. Since the distribution of $X$ is atomless, we can use Lemma \ref{Prop:3.2} and \eqref{lem1} to write
\begin{align}
\mathbb{E}\left[\log X\right] & =\frac{1}{(n+6)2^{n}}\left(\mathbb{E}\left[\log\left(\prod_{k=1}^{2^{n}}\left(a_{n}^{k}X+b_{n}^{k}\right)\right)\cdot\mathbf{1}_{\left(X<1\right)}\right]+\mathbb{E}\left[\log\left(\prod_{k=1}^{2^{n}}\left(a_{n}^{k}X+b_{n}^{k}\right)\right)\cdot\mathbf{1}_{\left(X>1\right)}\right]\right)\nonumber \\
 & \leq\frac{1}{(n+6)2^{n}}\left(\mathbb{E}\left[\log\left(\prod_{k=1}^{2^{n}}\left(a_{n}^{k}+b_{n}^{k}\right)\right)\cdot\mathbf{1}_{\left(X<1\right)}\right]+\mathbb{E}\left[\log\left(X^{2^{n}}\prod_{k=1}^{2^{n}}\left(a_{n}^{k}+b_{n}^{k}\right)\right)\cdot\mathbf{1}_{\left(X>1\right)}\right]\right).\nonumber
\end{align}
Moreover, using \eqref{eq:cn_rec} and \eqref{eq:a} from Lemma \ref{Expectation-Identity1}, it follows that
\begin{align}
\mathbb{E}\left[\log X\right] & \leq\frac{1}{(n+6)2^{n}}\Big(\log(c_n)\cdot\mathbb{P}(X<1)+2^n\mathbb{E}\left[\log\left(X\right)\cdot\mathbf{1}_{\left(X>1\right)}\right]+\log(c_n)\cdot\mathbb{P}(X>1)\Big)\nonumber \\
 & =\frac{\log c_{n}}{(n+6)2^{n}}+\frac{2\mathbb{E}\left[\log X\right]}{n+6}.\label{eq:12}
\end{align}
Subtracting the last term on the right-hand side of \eqref{eq:12} from both sides while recalling \eqref{eq:lambda_log} leads to
\[
\lambda=\mathbb{E}\left[\log X\right]\leq\frac{\log c_{n} }{\left(n+4\right)2^{n}}.
\]

For the lower bound, we can repeat this same procedure using \eqref{lem2} and \eqref{eq:b} instead of \eqref{lem1} and \eqref{eq:a} to arrive at

\[
\mathbb{E}\left[\log X\right]\geq \frac{\log c_{n}}{(n+6)2^{n}}-\frac{\mathbb{E}\left[\log X\right]}{n+6}.
\]
Similarly, this implies
\[
\frac{\log c_{n}}{(n+7)2^{n}}\leq\lambda.
\]

We now show that these bounds converge to the Lyapunov exponent as $n\to\infty$. We first point out the crude estimate $c_{n}\leq \left(F_{n+4}\right)^{2^{n}}$ where
$\{F_n\}:= \{f_{n}(0,1) \}$ is the usual Fibonacci sequence. This follows from \eqref{eq:cn_rec}, \eqref{eq:dominate}, and the fact that $a_n^1=F_{n+3}$ for all $n\geq 0$. Also note that the well-known asymptotic 
\[
F_n\sim\frac{(\phi_1)^n}{\sqrt{5}}~\text{ as }~n\to\infty
\]
implies
\[
\lim_{n\to\infty}\frac{\log\left(\left(F_{n+4}\right)^{2^{n}}\right)}{(n+4)2^{n}}=\log\left(\phi_1\right).
\]
Hence we have
\begin{align}
\limsup_{n\to\infty}\left|q_{n}-p_{n}\right| 
 & =\limsup_{n\to\infty}\frac{3\log c_{n}}{\left(n+7\right)\left(n+4\right)2^{n}}\nonumber \\
 & \leq \lim_{n\to\infty}\frac{3\log\left(\left(F_{n+4}\right)^{2^{n}}\right)}{(n+7)\left(n+4\right)2^{n}}\nonumber \\
 & =0.\nonumber
\end{align}
Now the result follows from \eqref{Bounds}.

\end{proof}

We end this section with the following two remarks. 

\begin{remark}
There doesn't seem to be an obvious recursion among the $c_{n}$
values. In order to compute $c_{n}$ using its definition, we must consider
$2^{n}$ coefficient pairs. We are able to compute ${p_{25}\approx 0.204266}$ and ${q_{25}\approx 0.225397}$ but going beyond $n=25$ exceeds our computing power. After implementing a simple numerical scheme to compute $\mathbb{E}\left[\log X\right]$ using the CDF of $X$ from Theorem 5.2 of \cite{Goswami2004} along with \eqref{eq:b}, we expect that $\lambda\approx 0.2165$.

\end{remark}

\begin{remark}
The bounds in Lemma \ref{UBprop} from the general $p$ case are analogous to $p_0$ and $q_0$ from \eqref{ConvergingSequence} of the $\Bernoulli\left(\frac{1}{2}\right)$ model. While we can attempt to improve these bounds by mimicking the proof of Theorem \ref{thm:bounds}, unlike in that case, there doesn't appear to be a nice expression for the corresponding bounds $p_n$ and $q_n$ as $n$ gets larger. 
\end{remark}

%
%
%



\section{$\mathbf{\xi} \cdot \Cauchy$ Parameter Model}\label{xicauchy}
The parameter model studied in this section is based on the standard Cauchy distribution (that is, Cauchy with location $x_0 = 0$ and scale $\gamma = 1$). Recall that the probability density function of a $\Cauchy \left(x_0,\gamma\right)$ random variable with location $x_0\in\mathbb{R}$ and scale $\gamma>0$ is
\begin{equation}\label{eq:Cauchy_dens}
f(x)=\frac{1}{\pi \gamma \left(1+\left(\frac{x-x_0}{\gamma}\right)^2\right)},\ -\infty < x < \infty.
\end{equation}

Let $\mu_\xi$ be the probability measure on $\mathrm{GL}(2,\mathbb R)$ given by 
\begin{equation}\label{CauchyMatrix1}
\left(\begin{array}{cc}
\xi \epsilon & -1\\
1 & 0
\end{array}\right), \ \epsilon\sim \Cauchy \left(0,1\right), \ \xi\in\mathbb{\mathbb{R}}, \ \xi\neq 0.
\end{equation}
The fact that $\mu_\xi$ satisfies the hypotheses of Theorem \ref{furstenberg} can be seen through a similar analysis as done in the beginning of Section \ref{bernp} with some slight differences which we now point out. To verify hypothesis $(i)$, we can use the Frobenius matrix norm to arrive at $\mathbb E[\log^+\|Y_1\|]=\frac{1}{2}\int\log(2+\xi^{2}x^{2})f(x)dx$ where $f(x)$ is the density for $\mbox{Cauchy}\left(0,1\right)$. By elementary computations, this integral is seen to be finite for all $\xi$. Hypothesis $(ii)$ can be verified in the same manner as for the $\Bernoulli\left(p\right)$ model. Hypothesis $(iii)$ follows from the unbounded support of $\epsilon$. For hypothesis $(iv)$, we can again use the equivalent condition given in \cite[Proposition II.4.3]{bougerol}. More specifically, draw $M$ from \eqref{CauchyMatrix1} with $\epsilon=\frac{1}{\xi}$ and proceed as in the beginning of Section \ref{bernp}. 

Hence we know there exists a unique $\mu_\xi$-invariant distribution $\nu_\xi$ such that a random variable $X_\xi$ has law $\nu_\xi$ if and only if it satisfies the distributional identity 
\begin{equation}\label{Cauchy-distributional}
 X_\xi\sim-\frac{1}{X_\xi}+\xi\epsilon,
\end{equation}
where $\epsilon\sim \Cauchy \left(0,1\right)$ and is independent of $X_\xi$. The goal of this section is to find the explicit value of the Lyapunov exponent $\lambda (\xi)$ related to $\mu_\xi$. Following the method from \cite[pp. 35]{bougerol}, we have an explicit formula for the Lyapunov exponent in terms of the parameter $\xi$. This will allow us to to study the variance in the Central Limit Theorem related to the products of random matrices of the form \eqref{CauchyMatrix1} as formulated in Theorem \ref{lepageclt}. Since the Lyapunov exponent used in our Monte Carlo simulation scheme will be exact, we can obtain a better approximation for the variance compared to the other models we study.  

\begin{prop}\label{CauchyProp}
Let $\mu_\xi$ be the probability measure on $\mathrm{GL}(2,\mathbb R)$ given by \eqref{CauchyMatrix1}. Then the Lyapunov exponent $\lambda(\xi)$ associated with $\mu_\xi$ is given by
\[
\lambda(\xi)=\log\left(\frac{|\xi|+\sqrt{\xi^{2}+4}}{2}\right).
\]
\end{prop}

\begin{proof} 
According to \eqref{expectation}, we have $\lambda (\xi)=\mathbb{E}\big[\log |X_\xi|\big]$, where $X_\xi$ is a random variable satisfying \eqref{Cauchy-distributional}. To find the law of such an $X_\xi$, we first guess that it is $\Cauchy \left(0,\gamma\right)$ for some $\gamma>0$ and then verify that it satisfies \eqref{Cauchy-distributional} for a particular $\gamma$.

Assuming that $X_\xi\sim\Cauchy \left(0,\gamma\right)$, the well-known transformation properties of the Cauchy distribution imply that the right-hand side of \eqref{Cauchy-distributional} is also Cauchy distributed, namely 
$$-\frac{1}{X_\xi}+\xi\epsilon \sim \Cauchy\left(0,\frac{1}{\gamma}+\left|\xi\right|\right).$$
Hence \eqref{Cauchy-distributional} holds if and only if
\[
\gamma=\frac{1}{\gamma}+\left|\xi\right|
\]
which has as its unique positive solution
\[
\gamma=\frac{|\xi|+\sqrt{\xi^2+4}}{2}.
\]
Now we can use \eqref{eq:Cauchy_dens} to write
\begin{align*}
\lambda(\xi)=\int_{-\infty}^{\infty}\log|x|\frac{1}{\pi \gamma\left(1+\left(\frac{x}{\gamma}\right)^2\right)}\,\mathrm d x&=\log(\gamma) \\
&=\log\left(\frac{|\xi|+\sqrt{\xi^{2}+4}}{2}\right).
\end{align*}
The proof is complete because of the uniqueness of the distribution $v_\xi$ such that \eqref{Cauchy-distributional} is satisfied. 
\end{proof}

Figure \ref{fig:Cauchy20Lambda} shows the graph of $\lambda(\xi)$ for $\xi\in[-20,20]$; 
in Figure \ref{fig:CauchyZoomLambda}, we plot $\lambda(\xi)$ for $\xi\in[-1,1].$

\begin{figure}[H]
\centering
\begin{subfigure}{.5\textwidth}
  \centering
\includegraphics[scale=0.5]{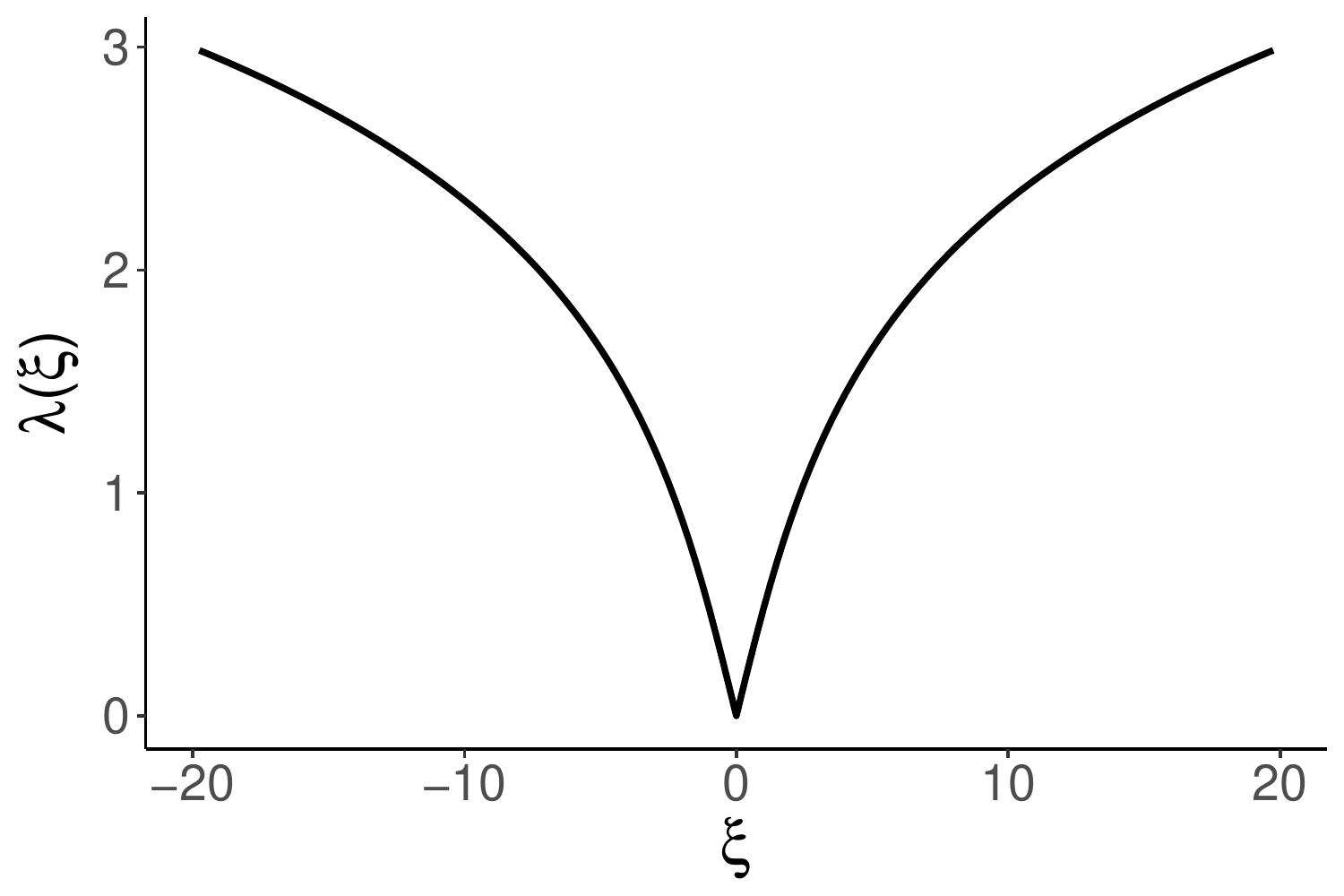}
\caption{}
\label{fig:Cauchy20Lambda}
\end{subfigure}%
\begin{subfigure}{.5\textwidth}
\centering
\includegraphics[scale=0.5]{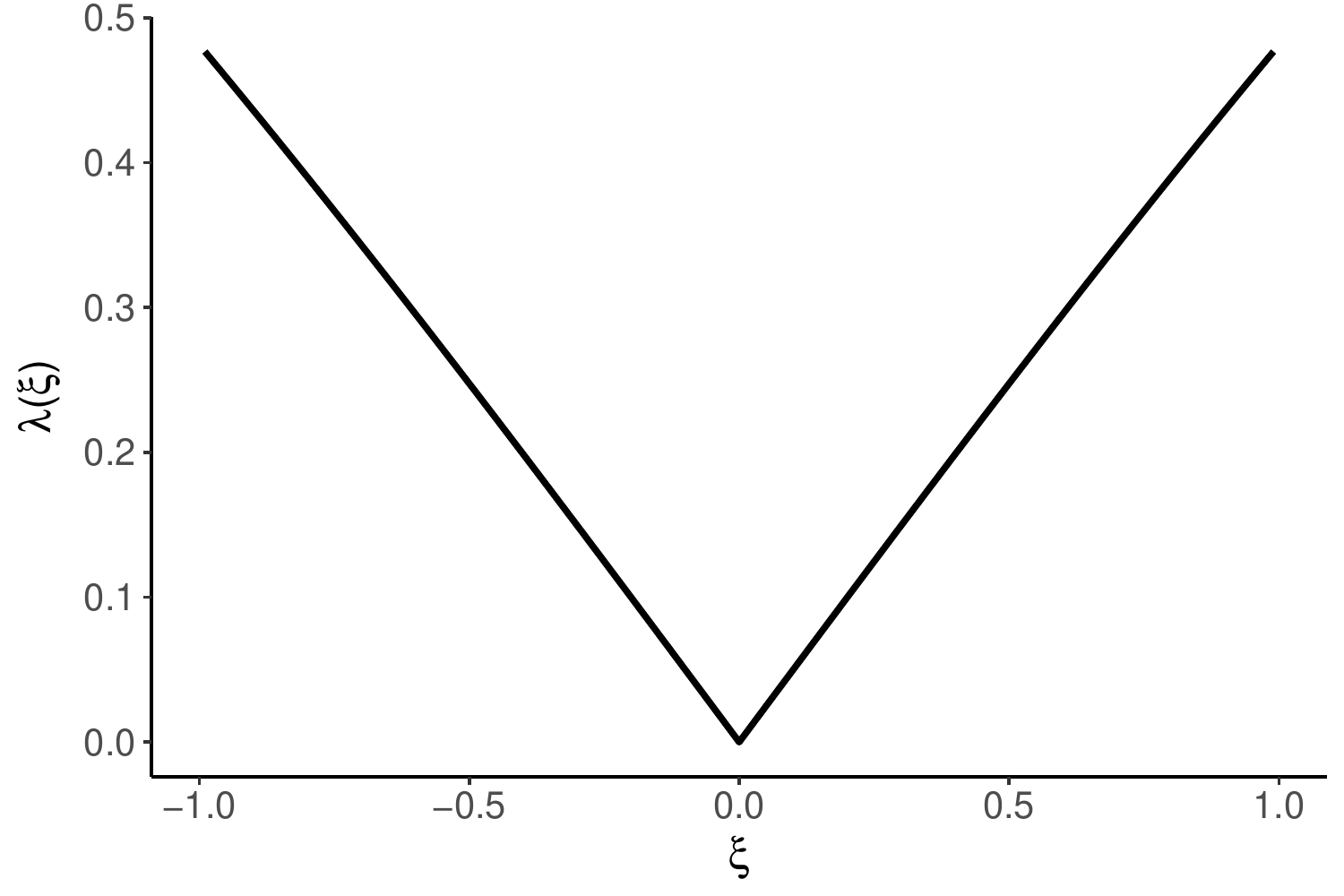}
\caption{}
\label{fig:CauchyZoomLambda}
\end{subfigure}
\caption{$\lambda(\xi)$ vs. $\xi$}
\label{fig:test}
\end{figure}


\section{Variance Simulation}\label{variance}

It is straightforward to verify that the hypotheses of Theorem \ref{lepageclt} are satisfied for the models we studied in Sections \ref{bernp} and \ref{xicauchy}. In fact, much of the reasoning done in the beginning of Sections \ref{bernp} and \ref{xicauchy} to verify the conditions of Theorem \ref{furstenberg} can be used to verify those of Theorem \ref{lepageclt}. For example, in the $\Bernoulli \left(p\right)$ model, hypothesis $(i)$ follows from the finite support of $\mu_p$. For the Cauchy model, we can again use the Frobenius matrix norm to see that $\mathbb{E}\left[\exp\left(t\ell\left(Y_{1}\right)\right)\right]=\int\left(2+\xi^{2}x^{2}\right)^{t/2}f(x)dx$ where $f(x)$ is the density for $\mbox{Cauchy}\left(0,1\right)$. By elementary computations, this integral is seen to be finite when $t<1$ and hence hypothesis $(i)$ is also satisfied for this model. Moreover, hypothesis $(ii)$ has already be verified for both models and hypothesis $(iii)$ follows from conditions $(ii)$ and $(iii)$ of Theorem \ref{furstenberg} which have already been verified. 

Thus for $0<p<1$ and $\xi\neq 0$, we know there exists $\sigma(p),\sigma(\xi)>0$ such that for any $\mathbf{x}\in\mathbb R^2\setminus\{\mathbf{0}\}$, 
\[
\frac1{\sqrt n}\Big(\log\|S_n\mathbf{x}\|-n\lambda(p)\Big)\textrm{ and }\frac1{\sqrt n}\Big(\log\|S_n\mathbf{x}\|-n\lambda(\xi)\Big)
\]
converge weakly as $n\to\infty$ to Gaussian random variables with mean $0$ and variance $\sigma^2(p)$ and $\sigma^2(\xi)$. Here the $S_n$ are products of matrices distributed according to the probability measures $\mu_p$ and $\mu_\xi$ given in Sections \ref{bernp} and \ref{xicauchy}, respectively.

Motivated by these considerations and following the idea of Section \ref{sec:approx}, we can approximate $\sigma^2(p)$ and $\sigma^2(\xi)$ by computing the variance of 
\[
L_p:=\frac{1}{\sqrt{n}}\left(\sum_{i=1}^{n}\log\left\|U_i\right\|-n\lambda(p)\right)\text{ and }~
L_\xi:=\frac{1}{\sqrt{n}}\left(\sum_{i=1}^{n}\log\left\|U_i\right\|-n\lambda(\xi)\right)
\]
with $n$ large. Here, as in Section \ref{sec:approx}, the sequence $\{U_i\}_{i\geq 0}$ is constructed recursively by $U_0=\mathbf{x}$ and $U_i=Y_i\frac{U_{i-1}}{\|U_{i-1}\|}$ for some $\mathbf{x}\in\mathbb{R}^2$ with $\|\mathbf{x}\|=1$ and $\{Y_i\}_{i\ge1}$ an i.i.d. sequence drawn from $\mu_p$ or $\mu_\xi$ as appropriate. While we have an exact expression for $\lambda(\xi)$, we must settle for the approximation of $\lambda(p)$ obtained by simulation in Section \ref{sec:approx}.

In what follows, we summarize the simulation procedure for $\sigma^2(p)$. The procedure for $\sigma^2(\xi)$ is practically identical.
\begin{enumerate}
\item Choose an interval $[a,b]$ as the range of $p$. Divide this interval
into sub-intervals of length $k$ where $k$ divides $b-a$. Let $p$ be of the form $a+jk$ for $j=0,1,\dots,\frac{b-a}{k}$. 
\item Choose a unit vector $\mathbf{x}\in\mathbb{R}^2$.
\item Simulate $L_p$ for each $p$ from Step 1 and store the result as a data vector of length $\frac{b-a}{k}+1$.
\item Repeat Step 3 an $m$ number of times to obtain an $m\times\frac{b-a}{k}+1$
matrix, where the $j^{\text{th}}$ column contains all of the $L_p$ simulations corresponding to $p=a+(j-1)k$.
\item Estimate $\Var\left(L_{a+(j-1)k}\right)$ by the sample variance of the $j^{\text{th}}$ column of the matrix.
\end{enumerate}
Note that in all of our simulations, we set
$\mathbf{x}=\left(\frac{\sqrt{2}}{2},\frac{\sqrt{2}}{2}\right)$ in Step 2.

We first approximate the variance for the $\Bernoulli \left(p\right)$ model considered in Section \ref{bernp}. Trivially, we have that ${\sigma^2(0)=\sigma^2(1)=0}$. For $0<p<1$, we simulate $\Var\left(L_p\right)$ with $k=0.01$, $n=1000$, and $m=1\,000\,000$. We plot the resulting points in Figure \ref{fig:BernPVars} and remark that the graph exhibits distinct asymmetry with the maximum variance occurring around $p=0.56$.

\begin{figure}[H]
\centering
\includegraphics[scale=0.6]{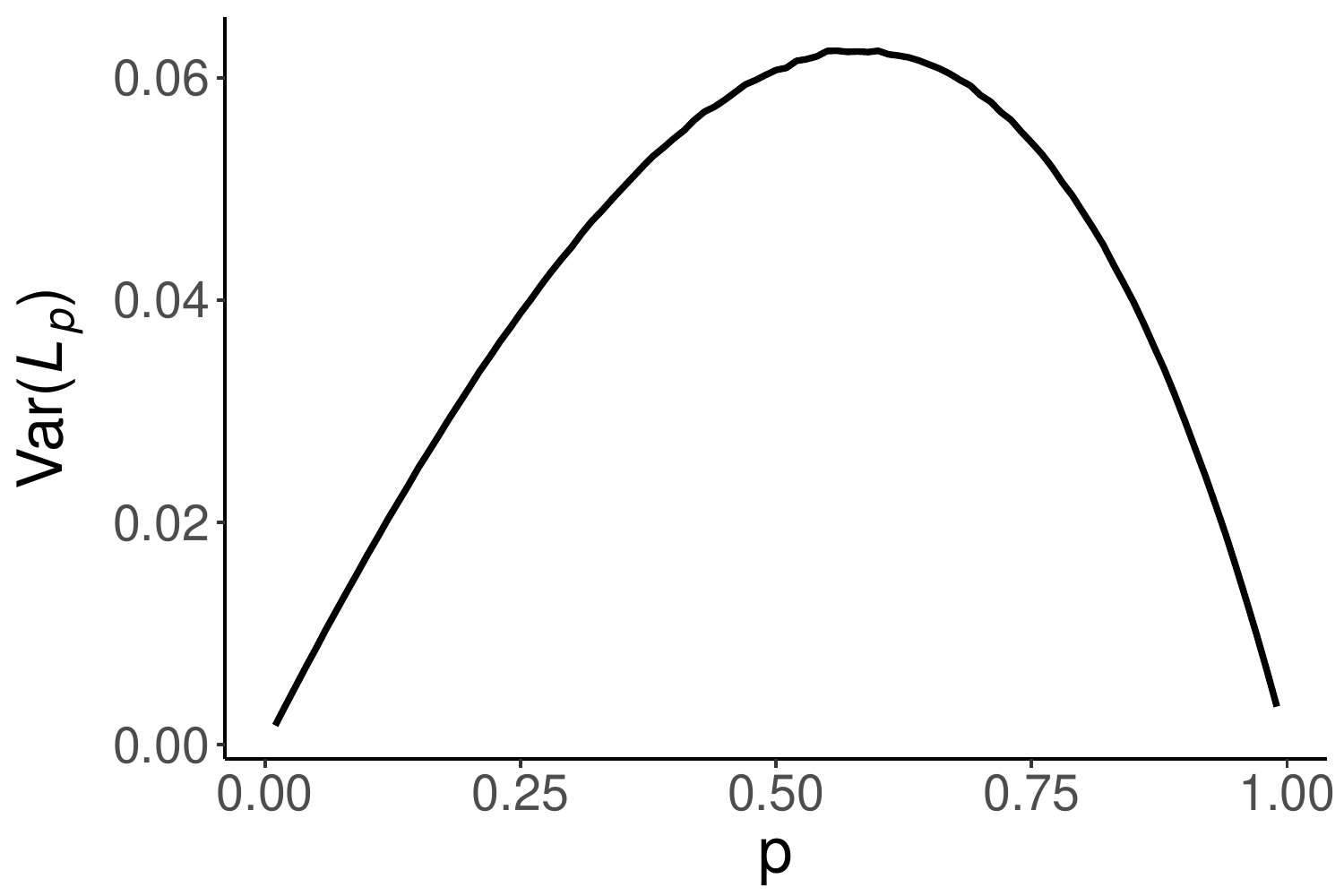}
\caption{$k=0.01$, $n=1000$, $m=1\,000\,000$}
\label{fig:BernPVars}
\end{figure}

For the Cauchy parameter model from Section \ref{xicauchy}, it is clear that $\sigma^2(0)=0$. For $\xi\neq 0$, we simulate $\Var\left(L_\xi\right)$ over both a large and small range of $\xi$. Figure \ref{fig:Cauchy20Vars} illustrates the results for 
$\xi\in[-20,20]$ with $k=0.25$. This is the same interval used to produce Figure \ref{fig:Cauchy20Lambda}.


\begin{figure}[H]
\centering
\includegraphics[scale=0.5]{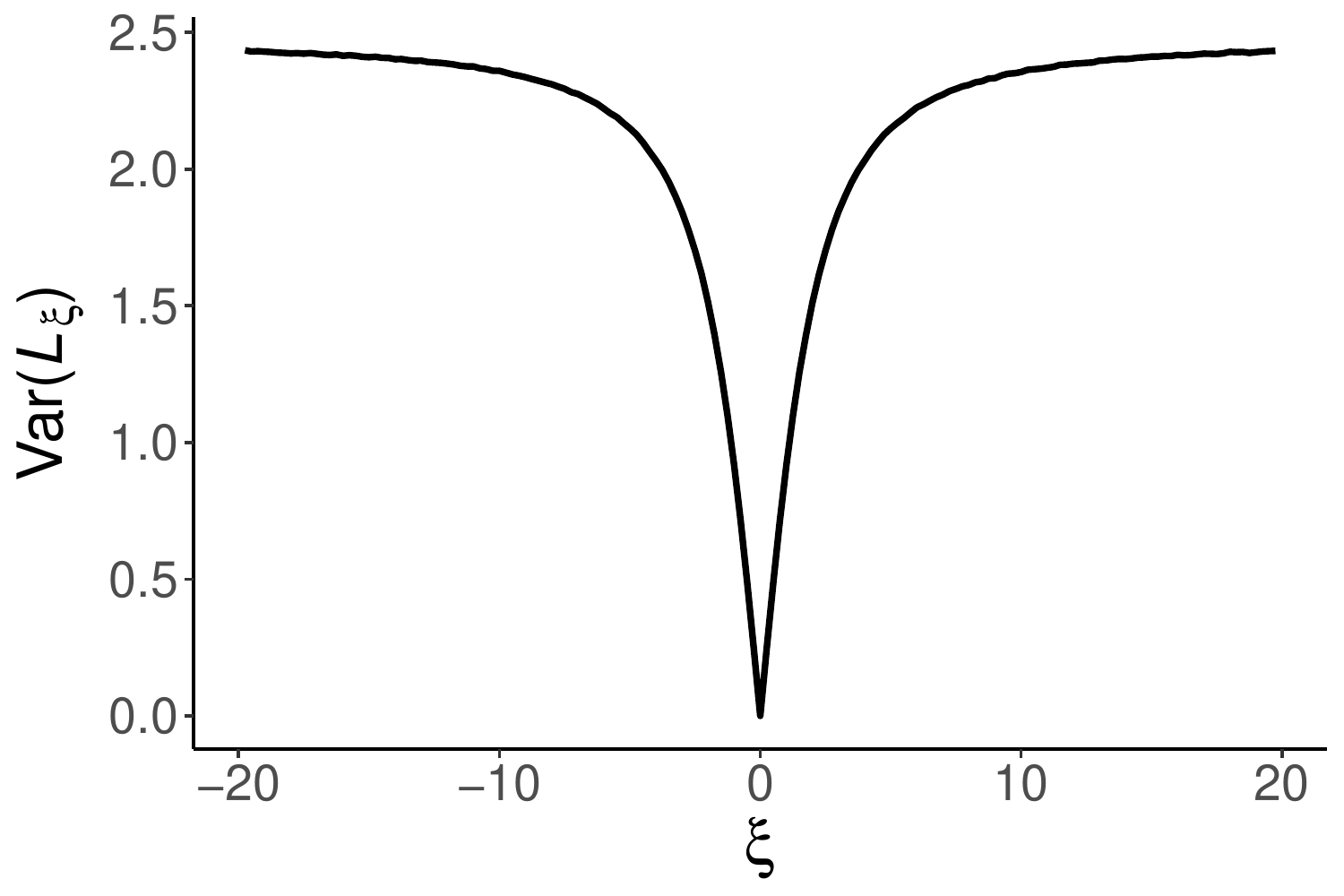}
\caption{$k=0.25$, $n=1000$, $m=5\,000\,000$}
\label{fig:Cauchy20Vars}
\end{figure}


In Figure \ref{fig:CauchyZoomVars}, we plot $\Var\left(L_\xi\right)$ for $\xi\in[-1,1]$ with $k=0.01$ to give a much finer resolution of the graph around the origin. 

\begin{figure}[H]
\centering
\includegraphics[scale=0.5]{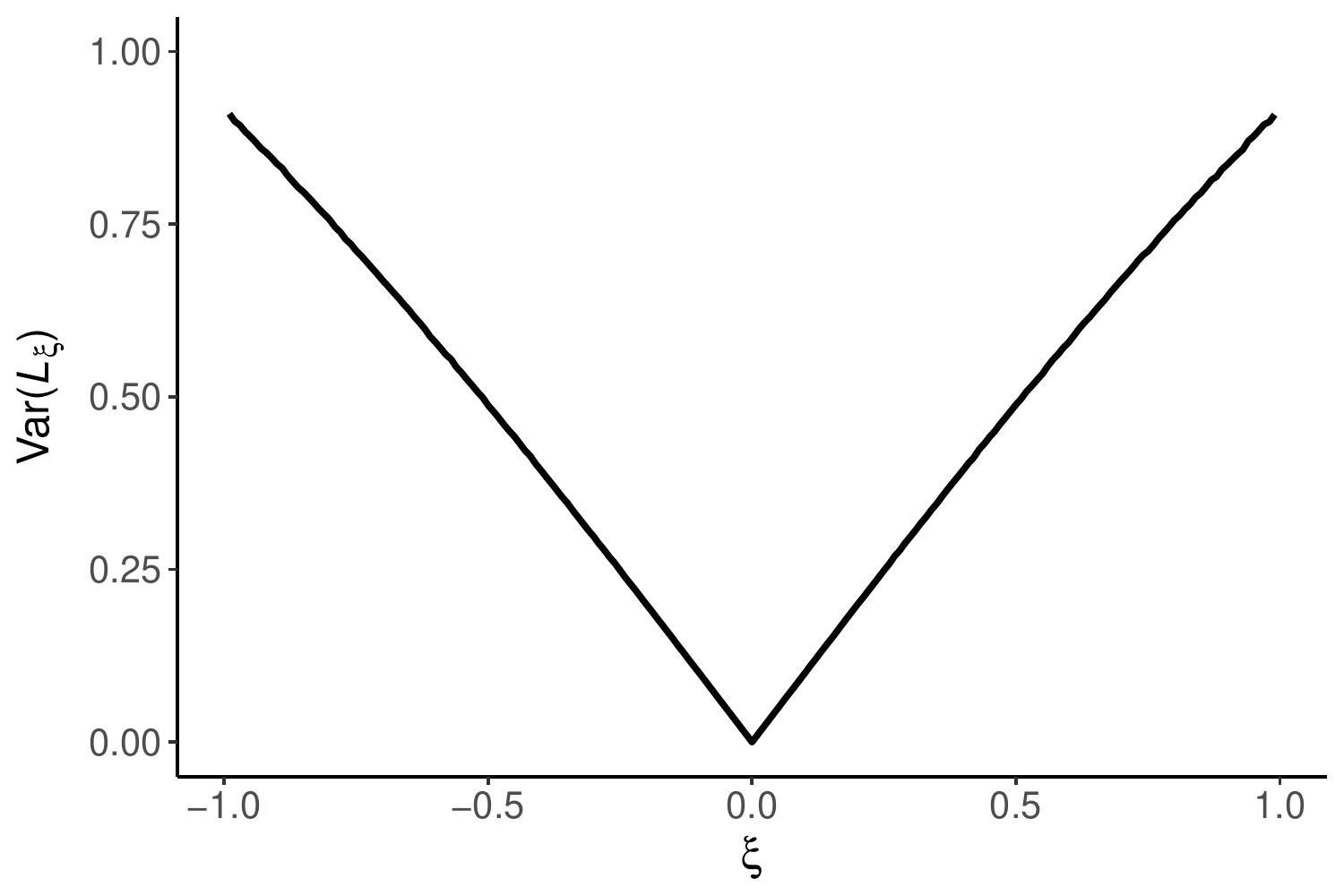}
\caption{$k=0.01$, $n=1000$, $m=1\,000\,000$}
\label{fig:CauchyZoomVars}
\end{figure}


\begin{acknowledgement}
The authors are grateful for many helpful and motivating conversations with M. Gordina, L. Rogers and A. Teplyaev. We would also like to thank an anonymous referee, whose comments and suggestions greatly improved the exposition of the paper. 
\end{acknowledgement}

\bibliography{lyapunov}

\providecommand{\bysame}{\leavevmode\hbox to3em{\hrulefill}\thinspace}
\providecommand{\MR}{\relax\ifhmode\unskip\space\fi MR }
\providecommand{\MRhref}[2]{%
  \href{http://www.ams.org/mathscinet-getitem?mr=#1}{#2}
}
\providecommand{\href}[2]{#2}
\begin{thebibliography}{10}

\bibitem{Akemann-Burda-Kieburg2014}
Gernot Akemann, Zdzislaw Burda, and Mario Kieburg, \emph{Universal distribution
  of {L}yapunov exponents for products of {G}inibre matrices}, J. Phys. A
  \textbf{47} (2014), no.~39, 395202, 35. \MR{3262164}

\bibitem{Akemann-Kieburg-Wei-2013}
Gernot Akemann, Mario Kieburg, and Lu~Wei, \emph{Singular value correlation
  functions for products of {W}ishart random matrices}, J. Phys. A \textbf{46}
  (2013), no.~27, 275205, 22. \MR{3081917}

\bibitem{GL_CLT}
Yves Benoist and Jean-Fran\c{c}ois Quint, \emph{Central limit theorem for
  linear groups}, Ann. Probab. \textbf{44} (2016), no.~2, 1308--1340.
  \MR{3474473}

\bibitem{bougerol}
Philippe Bougerol and Jean Lacroix, \emph{Products of random matrices with
  applications to {S}chr\"odinger operators}, Progress in Probability and
  Statistics, vol.~8, Birkh\"auser Boston, Inc., Boston, MA, 1985. \MR{886674}

\bibitem{chassaingetal}
Philippe Chassaing, G\'erard Letac, and Marianne Mora, \emph{Brocot sequences
  and random walks in {${\rm SL}(2,{\bf R})$}}, Probability measures on groups,
  {VII} ({O}berwolfach, 1983), Lecture Notes in Math., vol. 1064, Springer,
  Berlin, 1984, pp.~36--48. \MR{772400}

\bibitem{newman}
Joel~E. Cohen and Charles~M. Newman, \emph{The stability of large random
  matrices and their products}, Ann. Probab. \textbf{12} (1984), no.~2,
  283--310. \MR{735839}

\bibitem{Forrester-2013}
Peter~J. Forrester, \emph{Lyapunov exponents for products of complex {G}aussian
  random matrices}, J. Stat. Phys. \textbf{151} (2013), no.~5, 796--808.
  \MR{3055376}

\bibitem{Forrester-2015}
\bysame, \emph{Asymptotics of finite system {L}yapunov exponents for some
  random matrix ensembles}, J. Phys. A \textbf{48} (2015), no.~21, 215205, 17.
  \MR{3353003}

\bibitem{furstenbergkesten}
H.~Furstenberg and H.~Kesten, \emph{Products of random matrices}, Ann. Math.
  Statist. \textbf{31} (1960), 457--469. \MR{0121828}

\bibitem{Furstenberg-Kifer1983}
H.~Furstenberg and Y.~Kifer, \emph{Random matrix products and measures on
  projective spaces}, Israel J. Math. \textbf{46} (1983), no.~1-2, 12--32.
  \MR{727020}

\bibitem{Goswami2004}
Alok Goswami, \emph{Random continued fractions: a {M}arkov chain approach},
  Econom. Theory \textbf{23} (2004), no.~1, 85--105, Symposium on Dynamical
  Systems Subject to Random Shock. \MR{2032898}

\bibitem{Janvresse-Rittaud-Rue2008}
\'Elise Janvresse, Beno\^it Rittaud, and Thierry de~la Rue, \emph{How do random
  {F}ibonacci sequences grow?}, Probab. Theory Related Fields \textbf{142}
  (2008), no.~3-4, 619--648. \MR{2438703}

\bibitem{Janvresse-Rittaud-Rue2009}
\bysame, \emph{Growth rate for the expected value of a generalized random
  {F}ibonacci sequence}, J. Phys. A \textbf{42} (2009), no.~8, 085005, 18.
  \MR{2525481}

\bibitem{Janvresse-Rittaud-Rue2010}
\bysame, \emph{Almost-sure growth rate of generalized random {F}ibonacci
  sequences}, Ann. Inst. Henri Poincar\'e Probab. Stat. \textbf{46} (2010),
  no.~1, 135--158. \MR{2641774}

\bibitem{Kargin-2014}
Vladislav Kargin, \emph{On the largest {L}yapunov exponent for products of
  {G}aussian matrices}, J. Stat. Phys. \textbf{157} (2014), no.~1, 70--83.
  \MR{3249905}

\bibitem{Kieburg-Kosters-2019}
Mario Kieburg and Holger K\"{o}sters, \emph{Products of random matrices from
  polynomial ensembles}, Ann. Inst. Henri Poincar\'{e} Probab. Stat.
  \textbf{55} (2019), no.~1, 98--126. \MR{3901642}

\bibitem{limarahibe}
R.~Lima and M.~Rahibe, \emph{Exact {L}yapunov exponent for infinite products of
  random matrices}, J. Phys. A \textbf{27} (1994), no.~10, 3427--3437.
  \MR{1282183}

\bibitem{mtwInvariant}
Jens Marklof, Yves Tourigny, and Lech Wolowski, \emph{Explicit invariant
  measures for products of random matrices}, Trans. Amer. Math. Soc.
  \textbf{360} (2008), 3391--3427.

\bibitem{Newman1986}
Charles~M. Newman, \emph{The distribution of {L}yapunov exponents: exact
  results for random matrices}, Comm. Math. Phys. \textbf{103} (1986), no.~1,
  121--126. \MR{826860}

\bibitem{Peres-1991}
Yuval Peres, \emph{Analytic dependence of {L}yapunov exponents on transition
  probabilities}, Lyapunov exponents ({O}berwolfach, 1990), Lecture Notes in
  Math., vol. 1486, Springer, Berlin, 1991, pp.~64--80. \MR{1178947}

\bibitem{Peres-1992}
\bysame, \emph{Domains of analytic continuation for the top {L}yapunov
  exponent}, Ann. Inst. H. Poincar\'{e} Probab. Statist. \textbf{28} (1992),
  no.~1, 131--148. \MR{1158741}

\bibitem{pollicott}
Mark Pollicott, \emph{Maximal {L}yapunov exponents for random matrix products},
  Invent. Math. \textbf{181} (2010), no.~1, 209--226. \MR{2651384}

\bibitem{Protasov-Jungers-2013}
V.~Yu. Protasov and R.~M. Jungers, \emph{Lower and upper bounds for the largest
  {L}yapunov exponent of matrices}, Linear Algebra Appl. \textbf{438} (2013),
  no.~11, 4448--4468. \MR{3034543}

\bibitem{Ruelle-1979}
D.~Ruelle, \emph{Analycity properties of the characteristic exponents of random
  matrix products}, Adv. in Math. \textbf{32} (1979), no.~1, 68--80.
  \MR{534172}

\bibitem{Viswanath2000}
Divakar Viswanath, \emph{Random {F}ibonacci sequences and the number
  {$1.13198824\ldots$}}, Math. Comp. \textbf{69} (2000), no.~231, 1131--1155.
  \MR{1654010}

\end{thebibliography}
\bibliographystyle{amsplain}


\end{document}